\documentclass[a4paper,10pt]{amsart}

\usepackage[latin1]{inputenc}
\usepackage[T1]{fontenc}
\usepackage[english]{babel}

\usepackage{mathrsfs}
\usepackage{amscd}
\usepackage{amsfonts}
\usepackage{amsmath}
\usepackage{amssymb}
\usepackage{amstext}
\usepackage{amsthm}
\usepackage{amsbsy}

\usepackage{xspace}
\usepackage[all]{xy}
\usepackage{graphicx}
\usepackage{url}
\usepackage{latexsym}

\usepackage{ifthen} 

\makeatletter
\newcommand{\ominimal}{%
\ifthenelse{\equal\f@family{cmr}}
       {\ifthenelse{\equal\f@shape{n}}
               {o\kern0.05em-\kern-0.05emminimal}
               {\ifthenelse{\equal\f@shape{sl}}
                       {o\kern0.05em-\kern-0.05emminimal}
                       {\ifthenelse{\equal\f@shape{sc}}

{o\kern0.04em\raise0.1ex\hbox{-}\kern-0.04emminimal}
                                }}}
       {o-minimal} }

\newcommand{\ominimality}{%
\ifthenelse{\equal\f@family{cmr}}
       {\ifthenelse{\equal\f@shape{n}}
               {o\kern0.05em-\kern-0.05emminimality}
               {\ifthenelse{\equal\f@shape{sl}}
                       {o\kern0.05em-\kern-0.05emminimality}
                       {\ifthenelse{\equal\f@shape{sc}}

{o\kern0.04em\raise0.1ex\hbox{-}\kern-0.04emminimality}
                               {o-minimality}}}}
       {o-minimality} } 
 
\newcommand{\Ominimal}{%
\ifthenelse{\equal\f@family{cmr}}
       {\ifthenelse{\equal\f@shape{n}}
               {O-\kern-0.06em\kern0.5ptminimal}
               {\ifthenelse{\equal\f@shape{sl}}
                       {O-\kern-0.05emminimal}
                       {\ifthenelse{\equal\f@shape{sc}}
                               {O\raise0.1ex\hbox{-}\kern-0.04emminimal}
                               {O-minimal}}}}
       {O-minimal}}
\newcommand{\OMINIMAL}{%
       {O\kern0.04em\raise0.25ex\hbox{-}\kern-0.04emMINIMAL}}
\makeatother
\newtheorem{theo}{Theorem}[section]
\newtheorem{lem}[theo]{Lemma}
\newtheorem{prop}[theo]{Proposition}
\newtheorem{cor}[theo]{Corollary}
\newtheorem{fact}[theo]{Fact}

\theoremstyle{definition}
\newtheorem{dfn}[theo]{Definition}
\newtheorem{rem}[theo]{Remark}
\newtheorem{ex}[theo]{Example}
\newtheorem{nota}[theo]{Notation}
\newcommand{\N}{\ensuremath{\mathbb{N}}}

\newcommand{\R}{\ensuremath{\mathbb{R}}}

\newcommand{\M}{\ensuremath{\mathcal{M}}}
\newcommand{\Rs}{\ensuremath{\mathcal{R}}}

\newcommand{\g}{\ensuremath{\mathfrak{g}}}

\newcommand{\la}{\ensuremath{\mathfrak{a}}}

\newcommand{\n}{\ensuremath{\mathfrak{n}}}

\newcommand{\Rpl}{\ensuremath{M_{\oplus}}}
\newcommand{\Rx}{\ensuremath{M_{\otimes}}}

\newcommand{\vs}{\vspace{0.2cm}}

\DeclareMathOperator{\Aut}{Aut}

\DeclareMathOperator{\Ad}{Ad}
\DeclareMathOperator{\ad}{ad}
\DeclareMathOperator{\End}{End}

\DeclareMathOperator{\GL}{GL} 
\DeclareMathOperator{\SL}{SL} 
\DeclareMathOperator{\SO}{SO} 

\DeclareMathOperator{\card}{card}

\DeclareMathOperator{\Lie}{Lie}
\DeclareMathOperator{\id}{id}

\renewcommand{\geq}{\geqslant}
\newcommand{\book}[2]{{\scshape#1}, {\bf #2}}
\newcommand{\pre}[2]{{\scshape#1}, #2}
\newcommand{\publ}[6]
{{\scshape#1}, #2, {\itshape #3}, {\bf #4} (#5), pp.~#6.}
\begin{document}

\author{Annalisa Conversano} 
 
\title[Lie-like decomposition of definable groups]{Lie-like decompositions of 
groups \\ definable in \OMINIMAL\ structures}

\address{Department of Mathematics and Statistics, University of Konstanz, D-78457, Konstanz, Germany}

\email{annalisa.conversano@uni-konstanz.de}

\urladdr{http://www.math.uni-konstanz.de/~conversa/homeng.html} 

\date{\today}
\maketitle

\begin{abstract}
There are strong analogies between groups definable in \ominimal\ structures and real Lie groups.
Nevertheless, unlike the real case, not every definable group has maximal definably compact subgroups.
We study definable groups $G$ which are not definably compact showing that they
have a unique maximal normal definable torsion-free subgroup $N$; the quotient $G/N$ always has maximal definably compact subgroups, and for every such a $K$ there is a maximal definable torsion-free subgroup $H$
such that $G/N$ can be decomposed as $G/N = KH$, with $K \cap H = \{e\}$. 
Thus $G$ is definably homotopy equivalent to $K$. When $G$ is solvable then $G/N$ is already definably compact.  

In any case (even when $G$ has no maximal definably compact subgroup) we find a definable Lie-like decomposition
of $G$ where the role of maximal tori is played by maximal $0$-subgroups. 
\end{abstract}

\tableofcontents 

\section{Introduction}

Since Pillay proved that every $n$-dimensional group definable in an \ominimal structure with universe $M$ is a topological group locally definably homeomorphic to $M^n$ (\cite{pi1}), many other analogies with real Lie groups have been shown.

Especially in the compact case the algebraic structure of the groups in the two 
categories is very similar. For instance the torsion subgroup of abelian groups is the same (\cite{eo, pet07}), every solvable definably compact definably connected group is abelian (\cite{ps00}), every definably connected definably compact group is the union of the conjugates of any maximal
definable torus (\cite{zero}, \cite{div}) and the derived subgroup of a definably connected definably compact group equals the set of its commutators (\cite{hpp2}).

Moreover, as Pillay first conjectured (\cite{pi2}), for every definably compact group $G$ there is a canonical surjective homomorphism $\pi \colon G \to G/G^{00}$ into a compact real Lie group $G/G^{00}$ (\cite{pi2, bopp})
with the same dimension (\cite{hpp, pet07}), with the same homotopy type (\cite{ba, bm}) and 
elementarily equivalent to $G$ in the group structure (\cite{hpp2}).

On the other hand, in the non-compact case some outstanding differences come out.
If a connected real Lie group is non-compact, it contains maximal compact subgroups and admits the following decomposition:

\begin{theo} \emph{(Iwasawa, \cite[Theorem 6]{iwa}).} \label{complie}
Let $L$ be a connected real Lie group. Then $L$ has maximal compact subgroups, all connected and conjugate to each other. If $K$ is one of them, there are subgroups $H_1, \dots, H_s$ of $L$, all isomorphic to $\R$, such that any $x \in L$ can be decomposed uniquely and continuously in the form

\[
x = kh_1 \cdots h_s, \qquad \qquad k \in K,\ h_i \in H_i.
\]

\vs \noindent
In particular $L$ is homeomorphic to $K \times \R^s$.
\end{theo}

\vs
 
Note that if $L$ is solvable then $K$ is abelian, so it is a maximal torus of $L$.\\

We remark that in the literature the ``Iwasawa decomposition'' in general refers to a decomposition of semisimple groups (see \ref{iwasublie}) which is an intermediate step for Theorem \ref{complie}.\\

A fundamental result in order to obtain an \ominimal analogue of \ref{complie} is the following theorem of Peterzil and Steinhorn: if a definable group is not definably compact then it contains a $1$-dimensional definable torsion-free subgroup (\cite{ps}). 

However, unlike the real case, definable groups may not have maximal definably compact subgroups. One can see for instance \cite[5.6]{ps} or Strzebonski's example \cite[5.3]{str}, which is a $2$-dimensional abelian definably connected group $G$ defined in the field of the real numbers with infinitely many torsion elements but with no infinite definably compact subgroup.

Strzebonski's group $G$ is a $0$-group, namely it has the property that for every proper definable subgroup $K$ (including $K = \{e\}$), $E(G/K) = 0$, where $E$ denotes the \ominimal Euler characteristic (see \cite{vdd} and section $3$). Since definable torsion-free groups $H$ are such that $E(H) =  \pm 1$ (\cite{str}), it follows that any definable $0$-group $G$ which is not definably compact cannot be decomposed definably as in \ref{complie}, otherwise $E(G/K) = E(H_1) \cdots E(H_s) = \pm 1$.  

Actually, we show that this is a general fact, in the sense that if a definably connected group $G$ cannot be decomposed as a product of two definable subgroups $K$ and $H$, where $K$ is definably compact and $H$ is torsion-free, then
it contains $0$-groups which are not definably compact.

More precisely, if a definably connected group $G$ has torsion, then $E(G) = 0$ (\cite{str, pps1}) and by DCC on definable subgroups it contains some $0$-subgroup (\cite{str}). Each $0$-subgroup is contained in a maximal one, called a $0$-Sylow (by analogy with $p$-Sylow subgroups in finite group theory), and all maximal $0$-subgroups are conjugate to each other (\cite{str}). Therefore
if a $0$-Sylow of a definable group $G$ is definably compact, all $0$-subgroups of $G$ are definably compact. 

We prove that in this case for every $0$-Sylow $A$ of $G$, there is a definable decomposition:
\[
G = KH \qquad K = \bigcup_{x \in K}A^x \qquad K \cap H = \{e\},
\]   

\noindent
where $K$ and $H$ are definably connected subgroups, with $K$ maximal definably compact and $H$ maximal torsion-free. Thus we have
a perfect analogy with the real case, where for every maximal torus $T$ of a connected Lie group $L$ there is a maximal compact subgroup $K$ of $G$ which is the union of the conjugates of $T$ in $K$ and there are subgroups $H_1, \dots, H_s$ as in \ref{complie}.

Otherwise, if $A$ is a $0$-Sylow of $G$ which is not definably compact, there is a definable decomposition (which generalizes the previous one)

\[
G = PH, \qquad P = \bigcup_{x \in P}A^x 
\] 

\noindent
where $P$ and $H$ are again definably connected subgroups with $H$ maximal torsion-free, but $P$ of course is not definably compact (it contains $A$ which is not definably compact)
and there is an infinite intersection $P \cap H$.  
It turns out that $P$ contains a complement $K$ of $H$ in $G$ (abstractly isomorphic to $P/(P \cap H)$ which is definably compact), but $K$ cannot be definable (since $A$ is a $0$-group).

An important tool is the existence of a unique maximal normal definable torsion-free subgroup $N$ in every definable group $G$ (\ref{unifree}). It is contained in each maximal definable torsion-free subgroup $H$ of the above Lie-like decomposition. In general $H$ is not normal in $G$ and so it contains $N$ properly. 
(Otherwise $H = N$).

If $G$ is solvable then $G/N$ is shown to be definably compact (\ref{solvquozcomp}) and the decomposition above becomes (\ref{prodsolv})
\[
G = AN,
\]  

\vs \noindent
where $A$ is any $0$-Sylow of $G$.

In general, if $N$ is the maximal normal definable torsion-free subgroup of a definably connected group $G$
and $A$ is a $0$-Sylow of $G$, then $A \cap N$ is the maximal definable torsion-free subgroup of $A$ 
and so the quotient $A/(A \cap N)$, which is definably isomorphic to the $0$-Sylow $AN/N$ of $G/N$, is 
definably compact. Then, as we saw above, all $0$-subgroups of $G/N$ are definably compact and it follows
that $G/N = \bar{K}\bar{H}$, where $\bar{K}$ is definably compact and $\bar{H}$ is torsion-free. Moreover definable torsion-free groups are definably contractible (\cite{ps05}), and therefore $\bar{K}$ is definably homotopy equivalent to $G$ (\ref{homeonocomp}).
This shows that in the \ominimal context the homotopy type of a definably connected group is equal to the homotopy type of a definably compact group, as in the real case. The difference is that in the Lie case the compact group can be found always within the group, while in the \ominimal context it can be found always
in a canonical quotient (and within the group if and only if 
all its $0$-subgroups are definably compact).
 
In the following subsection and in the last section one can find a more detailed list of the main results. 
We remark that all of them build on previous theorems proved by many authors about groups definable in \ominimal structures, such as those contained in \cite{zero, bmo, bopp, solv, div, eo, hpp, hpp2, pps1, pps2, pps02, ps00, ps05, ps, pi1, pi2, str}, and on the well-known structure of real Lie groups.\\ 

\vs 
\textbf{Acknowledgments}. 
Most of the research leading to this paper was carryed out
during my Ph.D. at the University of Siena. 
I am deeply grateful to my thesis advisor Alessandro Berarducci 
for his unvaluable guidance. Together with him I thank 
Marcello Mamino, Margarita Otero, Ludovico Pernazza, Ya'acov Peterzil and
Anand Pillay for several enlightening conversations and explanation. 
I thank Marcello Mamino also for let me use his beautiful ``\ominimal \hspace{-0.16cm}'' typesetting and Margarita Otero for suggest me 
the problem to find an \ominimal analogue to the existence of 
maximal compact subgroups in Lie groups. 

Thanks to Alessandro Berarducci, Ludovico Pernazza and Margaret Thomas 
for commenting on a draft of this.


\subsection{The structure of the paper} In \textit{Section 2} we recall basic definitions and we give a proof of the well-known fact that every definable group has a solvable radical.\\

In \textit{Section 3} we prove that every definable group has a unique maximal normal definable torsion-free subgroup, and that every definable extension of a definable torsion-free group by a definably compact group
is definably isomorphic to their direct product.\\

In \textit{Section 4} we consider solvable groups $G$ proving that $G/N$ is definably compact, where $N$ is the maximal normal definable torsion-free subgroup of $G$.
As a consequence we get a definable Lie-like decomposition $G = AN$, where $A$ is any maximal $0$-subgroup of $G$. Moreover it turns out that
$G$ is always (abstractly) isomorphic to a definable semidirect product $N \rtimes G/N$, and the isomorphism is definable if and only if $A$ is definably compact. \\ 

In \textit{Section 5} we show a definable Iwasawa-like decomposition first for definably simple groups, and then for definably connected groups with definably compact solvable radical.\\

In \textit{Section 6} we use the above decomposition to study properties of the canonical projection
$\pi \colon G \to G/N$ by the maximal normal definable torsion-free 
subgroup $N$.

In particular it turns out that if $G$ is definably connected then
$G/N$ always has maximal definably compact 
subgroups (all definably connected and conjugate), and each of them (let us call it $K$) has a definable torsion-free complement $H$, i.e. $G/N = KH$
and $K \cap H = \{e\}$.\\

In \textit{Section 7} we show that if a definably connected group has definable Levi subgroups, then they are all conjugate to each other.\\

In \textit{Section 8} we consider definably connected subgroups $G$ of the general linear group $\GL_n(M)$.
They admit a definable decomposition $G = KH$ as a product of two definably connected subgroups $K$ and $H$, where $K$ is (maximal) definably compact and $H$ is (maximal) torsion-free.\\

In \textit{Section 9} we study definable exact sequences 

\[
1\ \longrightarrow\ N\ \longrightarrow\ G\ \longrightarrow\ K\ \longrightarrow\ 1  \\
\]

\vs
\noindent where $N$ is torsion-free and $K$ is definably connected and definably compact. 

It turns out that every
such an exact sequence splits abstractly (i.e. $G$ is abstractly isomorphic to a semidirect product $N \rtimes K$) and it splits definably if and only
if every $0$-subgroup of $G$ is definably compact. 

In particular, we find that if $K$ is semisimple then the sequence always splits definably. As a consequence of this last fact, we can show that every definable group has maximal semisimple definably connected definably compact subgroups (all conjugate).\\

In \textit{Section 10} we find a definable decomposition in the spirit of \ref{complie}, where maximal tori of
Lie groups are replaced by maximal $0$-subgroups (which are definable tori whenever they are definably compact).\\  

In \textit{Section 11} we consider the homotopy type of a definably connected group $G$ in an \ominimal expansion of a real closed field $M$, showing that $G$ is definably homeomorphic to $K \times M^s$, for every maximal definably compact subgroup $K$ of $G/N$, where $N$ is the maximal normal definable torsion-free subgroup of $G$ and $s \in \N$ is the maximal dimension of a definable torsion-free subgroup of $G$.\\

\textit{Section 12} is a brief summary of the main results of the paper, focusing on the structure of a definable group in terms of certain particular definable characteristic subgroups and their quotients. \\

\vs 
\section{Preliminaries}

We assume \M\ to be an \ominimal expansion of an ordered group. 
A basic reference for \ominimality is \cite{vdd}.
We remark that we need the assumption that the structure expands an ordered group only to use the decomposition in \ref{decomp}, so far proved in this setting by Hrushoski, Peterzil and Pillay in \cite{hpp2}. However we believe that Fact \ref{decomp} (and therefore all the results of this paper as well) holds in an arbitrary \ominimal structure.

By ``\textbf{definable}'' as usual we mean ``definable in \M\ with parameters''.

Many results about definable groups in \ominimal structures and some theorems about Lie groups are used. In order to make the reading as smooth as possible, we will recall the most important ones (naming them ``facts'') just before we first need them. One can see \cite{surv} for an overview about definable groups in \ominimal structures, and \cite{kna} for an overview about Lie groups.\\

We give now some notation and vocabulary (even though they are quite standard).\\
Let $H, K < G$ be groups, $X \subset G$ a set. For any $g \in G$, $K^g$ denotes the \textbf{conjugate} of $K$ by $g$, namely the subgroup made by the elements of the form $gxg^{-1}$, with $x \in K$.
\begin{align*}
Z(G) &= \{g \in G : gx = xg\ \forall\, x \in G\} \mbox{ is the \textbf{center} of } G.\\
C_H(X) &= \{g \in H : gx = xg\ \forall\, x \in X\} \mbox{ is the \textbf{centralizer} of } X \mbox{ in } H. \\
N_H(K) &= \{g \in H : K^g = K\} \mbox{ is the \textbf{normalizer} of } K \mbox{ in } H.\\
\end{align*}

Let $G$, $N$, $Q$ be definable groups.

A \textbf{definable extension of $Q$ by $N$} is a definable group $G$ containing $N$, together with a definable homomorphism $\pi \colon G \to Q$ such that the sequence (where $i \colon N \to G$ is the inclusion map)

\[
1\ \longrightarrow\ N\ \stackrel{i}{\longrightarrow}\ G\ \stackrel{\pi}{\longrightarrow}\ Q\ \longrightarrow\ 1
\] 

\vs \noindent

is exact, i.e. $\pi$ is a surjective definable homomorphism and $i(N) = N$ is the kernel of $\pi$.  

A \textbf{section} $s$ of a surjective homomorphism $\pi \colon G \to Q$ is a map $s \colon Q \to G$ such that $\pi \circ s =  id_Q$.
We say that a definable extension $\pi \colon G \to Q$ \textbf{splits abstractly}, if there is a section of $\pi$ which is
a homomorphism. We say that it \textbf{splits definably} if there is such a section which is moreover definable.

Thus if $N$ is a normal definable subgroup of a definable group $G$, we can always see $G$ like a definable extension of $G/N$ by $N$. The exact sequence

\[
1\ \longrightarrow\ N\ \stackrel{i}{\longrightarrow}\ G\ \stackrel{\pi}{\longrightarrow}\ G/N\ \longrightarrow\ 1  \\
\]

\vs \noindent
splits abstractly if and only if
$G$ contains a subgroup $K$ isomorphic to $G/N$ such that $N \cap K = \{e\}$ and $NK = G$, i.e. if and only if $G$ is a \textbf{semidirect product} of $N$ and $K$. In this case we will write $G = N \rtimes K$ or $G = K \ltimes N$ and we will say that $K$ is a \textbf{cofactor} of $N$. The exact sequence above splits definably if and only if there is a \textbf{definable cofactor} $K$ of $N$.\\ 


If $H, K < G$ are groups (no normality assumption here), we say that $H$ is a \textbf{complement} of $K$ in $G$ whenever  $G = KH$ and $K \cap H = \{e\}$.  

If all groups are definable, then we say that $H$ is a \textbf{definable complement} of $K$ in $G$.\\

As usual, we denote by $G^0$ the definably connected component of the identity in a definable group $G$.
It is the smallest definable subgroup of finite index in $G$ (\cite[2.12]{pi1}). 
We say that a group $G$ is \textbf{definably connected} if it is definable and $G = G^0$, so if $G$ has no 
proper definable subgroup of finite index.\\

A group $G$ is \textbf{definably compact} if it is definable and every definable curve in $G$ is completable in $G$. (\cite[1.1]{ps}). Saying that a group $G$ \textbf{is not definably compact}, we assume that it is definable.\\

Let $\Rs = \langle R <, +, \cdot \rangle$ be a real closed field. 
If we say that a group is \textbf{semialgebraic over \Rs}, we mean that it is definable in the structure $\langle R, +, \cdot \rangle$.\\

If we say that a definable group $G$ is \textbf{linear}, we mean that there is a real closed field \Rs\ definable in \M\ such that $G$ is definably isomorphic to a subgroup of $\GL_n(\Rs)$ definable in \Rs, for some $n \in \N$. With an abuse of notation sometimes we will write $G < \GL_n(\M)$, as the whole structure expands a real closed field (even if we will not need this assumption), and identifying $G$ with its isomorphic image in $\GL_n(\Rs)$.\\

A definable group $G$ is said to be {\bf semisimple} if it is infinite with no
infinite abelian normal (definable) subgroup. 

It is \textbf{definably simple} if it is infinite and non-abelian with no proper non-trivial normal definable subgroup.

A fundamental theorem due to Peterzil, Pillay and Starchenko about semisimple definable groups is the following:

\begin{fact} \emph{(\cite[4.1]{pps1}, \cite[5.1]{pps02}).} \label{theosem}
Let $G$ be a semisimple definably connected group. Then there are definable real closed fields $\Rs_i$ such that $G/Z(G)$ is definably isomorphic to a direct product $H_1 \times \dots \times H_s$, where for every $i = 1, \dots, s$, $H_i$ is a semialgebraic $($over $\Rs_i)$ definably simple subgroup of
$\GL(n_i, \Rs_i)$. 
\end{fact}
\begin{rem} \label{semisnosolv}
An infinite definable group is semisimple if and only if it has no infinite solvable normal definable subgroup.
\end{rem}

\begin{proof} \hspace{7cm}
\begin{enumerate}

\item[$(\Rightarrow)$] Let $G$ be a semisimple definable group. We can suppose $G$ definably connected, because 
$G$ is semisimple if and only if $G^0$ is semisimple. If $G$ contains an infinite normal solvable definable subgroup, then its image in $G/Z(G)$ by the canonical projection is an infinite normal solvable definable subgroup, in contradiction with Fact \ref{theosem}.\\

\item[$(\Leftarrow)$] Every abelian group is solvable.
\end{enumerate}
\end{proof}

We end the section with a well-known fact: the existence of the solvable radical in any definable group. 
Because we could not find a precise reference and we will use it extensively, we also give a proof of it.

\begin{fact}\label{solvrad}
Every definable group $G$ has a unique normal solvable definably connected subgroup $R$ such that $G/R$ is finite
or semisimple. The subgroup $R$ contains every normal solvable definably connected subgroup of $G$, so it 
is the unique maximal normal solvable definably connected subgroup of $G$. It is said to be \textbf{the solvable radical} of $G$.
\end{fact}

\begin{proof}
To reduce to the definably connected case, observe that if $R$ is the solvable radical of $G^0$ then $R$ is also the solvable radical of $G$, since $G/G^0$ is finite and a finite extension of a semisimple definable group is semisimple.

First we prove that a non-solvable definably connected group $G$ has a normal solvable definably connected subgroup $R$ such that $G/R$ is semisimple. 

We proceed by induction on $n = \dim G$.
If $n = 1$, then $G$ is abelian (\cite[2.15]{pi1}). So let $n > 1$. If $G$ is not already semisimple,
let $A < G$ be an infinite abelian normal definable subgroup.   
If $G/A$ is semisimple then we set $R = A^0$. Otherwise we can apply induction on $G/A$
and find a normal solvable definably connected subgroup
$\bar{S} < G/A$ such that $(G/A)/\bar{S}$ is semisimple. 
Therefore we can take $R$ to be the definably connected component of the identity of the preimage of $\bar{S}$ in $G$.

We want now to show that every normal solvable definably connected subgroup $S$ of $G$ is contained in $R$. For this it is enough to note that $RS/R$ is a normal solvable definably connected subgroup of the semisimple definable group $G/R$. Since it cannot be infinite (\ref{semisnosolv}), it follows that $RS = R$. Hence $S \subseteq R$.

If $S \neq R$ then $\dim S < \dim R$ and $R/S$ is an infinite solvable definable subgroup of $G/S$. Therefore $G/S$ cannot be semisimple (\ref{semisnosolv}) and $R$ is the unique normal solvable definably connected subgroup of $G$ such that $G/R$ is semisimple.
\end{proof}

\begin{rem}
As we will use it in the sequel, let us observe that $Z(G)^0 \subseteq R$, 
since $Z(G)^0$ is an abelian normal definably connected subgroup of $G$.
\end{rem}

\vs
\section{The maximal normal definable torsion-free subgroup}

In this section we prove that: $(1)$ every definable group has a normal definable torsion-free subgroup (possibly trivial) which contains each normal definable torsion-free subgroup (\ref{unifree}), and $(2)$ every definable extension of a definable torsion-free group by a definably compact group is definably isomorphic to their direct product (\ref{comp_tor}). \\

We make use (here and later on) of the \ominimal Euler characteristic, an invariant by definable bijections which plays an important role in the study of definable groups. 

If $\mathcal{P}$ is a cell decomposition of a definable set $X$, \textbf{the \ominimal Euler characteristic} $E(X)$ is the integer defined as the number of even-dimensional cells in $\mathcal{P}$ minus the number of odd-dimensional cells in $\mathcal{P}$. This does not depend on $\mathcal{P}$ (see \cite[Chapter 4]{vdd}). 
When $X$ is finite then $E(X) = \card(X)$. Since 
for every $A, B$ definable sets, $E(A \times B) = E(A)E(B)$, the following holds:

\begin{fact} \emph{(\cite[2.12]{str}).} \label{euprodsub}
Let $K < H < G$ be definable groups. Then

\begin{enumerate}

\vspace{0.1cm}
\item[$(a)$] $E(G) = E(H)E(G/H)$.

\vspace{0.1cm}
\item[$(b)$] $E(G/K) = E(G/H)E(H/K)$.
\end{enumerate}
\end{fact}

\noindent
\begin{rem}
In order to have a quotient as a definable set, in general one assumes that the structure satisfies the definable choice property (as Strzebonski does in \cite{str}), namely that each definable equivalence relation on a definable set has a definable set of representatives (for instance, any \ominimal expansion of an ordered group satisfies the definable choice property: \cite[Chapter 6]{vdd}). In the case of a quotient of definable groups this is granted by Edmundo's theorem \cite[7.2]{solv}.
Since we always consider $E(X)$, where $X = G/H$ for some definable groups $H < G$, so far we can avoid any other assumption on the structure, but \ominimality \hspace{-0.12cm}. 
\end{rem}

\noindent
For every definable group $G$, if $p$ is a prime dividing $|E(G)|$, then $G$ has an element of order $p$ (\cite[2.5]{str}). Therefore:

\begin{fact} \emph{(\cite{str}).} \label{eu}  
Let $G$ be a definable group. Then 

\[
E(G) = \pm 1\  \Leftrightarrow\  G \mbox{ is torsion-free}.  
\]
 
\end{fact}  

\vspace{0.1cm}
\begin{fact} \emph{(\cite{str, pps1}).} \label{euconn}
Let $G$ be a definable group. If $G$ is definably connected, then 
$E(G) \in \{\pm 1, 0\}$. In particular, if $G$ is definably connected
and it is not torsion-free, then $E(G) = 0$.
\end{fact}

\noindent
We record now some properties about definable torsion-free groups from \cite{ps05}:

\begin{fact} \emph{(\cite[section 2]{ps05}).} \label{torfree} 
Let $H$ be a definable torsion-free group. Then:

\begin{enumerate}

\item[$(a)$]  $H$ is definably connected.  

\item[$(b)$]  $H$ is solvable.   

\item[$(c)$]  $H$ has a normal definable subgroup of codimension $1$.

\item[$(d)$]  Every definable quotient of $H$ is torsion-free.

 
\end{enumerate}

\end{fact}

\noindent
Using the facts just mentioned, we can prove that every definable group has a unique maximal normal definable torsion-free
subgroup:

\begin{theo}\label{unifree}
In every definable group $G$ there is a normal definable torsion-free
 subgroup $N$ which contains every
normal definable torsion-free subgroup of $G$. It is the unique normal definable torsion-free subgroup of $G$
of maximal dimension. We will refer to it as  
{\bf the maximal normal definable torsion-free subgroup} of $G$.
\end{theo}

\begin{proof}
Let $N$ be a normal definable torsion-free subgroup of $G$ of maximal dimension and $H$ a normal definable torsion-free subgroup of $G$.
We want to show that $H \subseteq N$.

We claim that $HN$ is a normal definable torsion-free subgroup of $G$: the definable group $H/(H \cap N)$ is a quotient of definable torsion-free groups, then it is torsion-free 
and it is definably isomorphic to $HN/N$. Thus 
$E(HN) = E(N)E(HN/N) = \pm1$ and $HN$ is torsion-free (\ref{eu}).

But $N$ is of maximal dimension among the normal definable torsion-free subgroups of $G$, so 
$\dim(HN) = \dim(N)$. 
Since definable torsion-free groups are definably connected, it follows that $HN = N$, $H \subseteq N$ and 
$\dim H < \dim N$, unless $H = N$.
\end{proof}

\begin{rem}\label{Nnotmax}
In general $N$ does not contain every definable torsion-free subgroup of $G$. For instance if $\M = \langle M, <, +, \cdot \rangle$
is a real closed field and $G = \SL_2(M) = \{A \in \GL_2(M) : \det A = 1\}$, then $N = \{e\}$ but $G$ does contain definable torsion-free subgroups such as 
\begin{displaymath}
H = \left\{
\left( \begin{array}{cc}
1 & a   \\
0 & 1    
\end{array} \right) : a \in M \right\},
\end{displaymath}

\vs \noindent
which is definably isomorphic to $(M, +)$. 

But if $G/N$ is definably compact, then $N$ contains every definable torsion-free subgroup of $G$:  
\end{rem}

\begin{lem}
A normal definable torsion-free subgroup $N$ of a definable group $G$ contains every definable torsion-free
subgroup of $G$ if and only if $G/N$ is definably compact.  
\end{lem}

\begin{proof}\hspace{7cm}
\begin{enumerate}
\item[$(\Rightarrow)$] If $G/N$ is not definably compact then by \cite[1.2]{ps} $G/N$ contains definable infinite torsion-free subgroups, and their preimages in $G$ are definable torsion-free subgroups containing $N$ properly.

\vs
\item[$(\Leftarrow)$] Let $H$ be a definable torsion-free subgroup of $G$. We want to show that $H \subseteq N$.
If not, $H/(H \cap N)$ is an infinite definable torsion-free group definably isomorphic to $HN/N$ which is
a definable subgroup of the definably compact group $G/N$, contradiction (see \ref{duacomptf}).
\end{enumerate}
\end{proof}

\begin{cor}\label{quocomp}
If $N$ is a normal definable torsion-free subgroup of a definable group $G$ such that $G/N$ is definably compact, then $N$ is the unique maximal definable torsion-free subgroup of $G$.
\end{cor}


\begin{lem}\label{radgsun}
Let $N$ be the maximal normal definable torsion-free subgroup of a definable group $G$. If $R$ is the solvable radical of $G$, then $N \subseteq R$ and $R/N$ is the solvable radical of $G/N$.
\end{lem}

\begin{proof}
Because $N$ is solvable and definably connected (\ref{torfree}), $N \subseteq R$.

Since $(G/N)/(R/N)$ is definably isomorphic to $G/R$ which is semisimple or finite, it follows that
$R/N$ is the solvable radical of $G/N$ (\ref{solvrad}).
\end{proof}

In the next section we will show that $R/N$ is always definably compact (\ref{solvquozcomp}), deducing that every definable group is a definable extension of a group with definably compact solvable radical by a (maximal) definable torsion-free group (\ref{extcomprad}).


\begin{lem}\label{tormaxnor}
Let $H \lhd\, G$ be definable groups. The maximal normal definable torsion-free
subgroup $N$ of $H$ is normal in $G$ as well.
\end{lem}

\begin{proof}
For every $g \in G$, the definable automorphism $H \to H$ given by $x \mapsto gxg^{-1}$, maps $N$ to $N^g$, 
that is a normal definable torsion-free subgroup of $H$ of the same dimension as $N$. Therefore 
$N^g = N$ for every $g \in G$ (\ref{unifree}).
\end{proof}



We show now that every definable extension of a definable torsion-free group by a definably compact
group is definably isomorphic to their direct product. 
The case where the torsion-free group is $1$-dimensional
was proved by Edmundo in \cite[5.1]{solv}. 

The specular case of
definable extensions of a definably compact group by a definable torsion-free group will be studied in section \ref{splext}.

\begin{theo}\label{comp_tor}
Let $\pi \colon G \to Q$ be a definable extension of a definable torsion-free group $Q$ by a definably compact group $K$. Then the definable exact sequence 

\[
1\ \longrightarrow\ K\ \stackrel{i}{\longrightarrow}\ G\ \stackrel{\pi}{\longrightarrow}\ Q\ \longrightarrow\ 1  \\
\]

\vs \noindent
splits definably in a direct product. Therefore the maximal normal definable torsion-free subgroup $N$ of $G$ is definably isomorphic to $Q$ and

\[
G\ =\ K \times N.\\
\]

\end{theo}

\begin{proof}
We proceed by induction on $n = \dim Q$.
For $n = 1$, see \cite[5.1]{solv}.
If $n > 1$, let $Q_1 \subset Q$ be a normal definable subgroup of $Q$ of codimension $1$ (\ref{torfree}) and let $G_1 = \pi^{-1}(Q_1)$. By induction, the definable exact sequence

\[
1\ \longrightarrow\ K\ \stackrel{i}{\longrightarrow}\ G_1\ \stackrel{\pi}{\longrightarrow}\ Q_1\ \longrightarrow\ 1  \\
\]

\vs \noindent
splits definably in a direct product. So there is a definable torsion-free subgroup $H$ normal in $G_1$ such that
$G_1 = K \times H$. Since $G_1/H \cong K$ which is definably compact, it follows by \ref{quocomp} that
$H$ is the maximal normal definable torsion-free subgroup of $G_1$ and by
\ref{tormaxnor} $H$ is normal in $G$ as well.
Consider now the definable exact sequence

\[
1\ \longrightarrow\ G_1/H \ \longrightarrow\ G/H\ \longrightarrow\ (G/H)/(G_1/H)\ \longrightarrow\ 1  \\
\]

\vs \noindent
with the obvious maps.
Because $G_1/H \cong K$ which is definably compact and 
$(G/H)/(G_1/H) \cong G/G_1 \cong Q/Q_1$ which is torsion-free, we can apply induction again
and find a definable torsion-free subgroup $S$ normal in $G/H$ such that $G/H = G_1/H \times S$. 
Therefore the preimage of $S$ in $G$ is a normal definable torsion-free subgroup which
is a direct complement to $K$ in $G$.
\end{proof}

\vs
\section{Solvable definable groups}

This section is devoted to the study of solvable definable groups. The main results are that the quotient of a solvable definable group by its maximal normal definable torsion-free subgroup is definably compact (\ref{solvquozcomp}) and that every solvable definably connected group can be decomposed as the product of its maximal normal definable torsion-free subgroup and any of its maximal $0$-subgroups (\ref{solvsempr}). Many consequences of these facts are deduced in this and in the subsequent sections.\\  

As we recalled in \ref{torfree}, definable torsion-free groups are solvable (\cite[2.12]{ps05}). On the opposite side there are solvable definably compact groups. Peterzil and Starchenko prove in \cite[5.4]{ps00} that they are abelian-by-finite, i.e. that every solvable definably connected definably compact group is abelian.

We study here the mixed case, where $G$ is a solvable definable group neither definably compact nor torsion-free, so it contains both definably compact and definable torsion-free subgroups with trivial intersection. Indeed: 

\begin{fact}\hspace{7cm} \label{compnoncomp}
\begin{enumerate}

\item[$(a)$] Every definable group which is not definably compact contains a definable $1$-dimensional torsion-free
subgroup \emph{(\cite[1.2]{ps})}. 

\vspace{0.1cm}
\item[$(b)$] Every definably compact group has torsion \emph {(\cite{hpp, pps1, pps02})}.

\end{enumerate}
\end{fact}

\begin{rem} \label{duacomptf}
It follows from \ref{compnoncomp}$(b)$ that a definable torsion-free group is not definably compact, and so definably compact
groups have no definable torsion-free subgroup. This is because definable subgroups are closed by \cite[2.8]{pi1}, so definable subgroups of definably compact groups are definably compact.

Therefore if $K, H < G$ are definable subgroups where $K$ is definably compact and $H$ is torsion-free, then $K \cap H = \{e\}$. It follows that if in addition
$G = KH$, then $K$ is a maximal definably compact subgroup and $H$ is a maximal torsion-free
definable subgroup.  
\end{rem}

\begin{theo}\label{solvquozcomp}
Let $G$ be a solvable definable group and let $N$ be its maximal normal definable torsion-free subgroup. If $G$ is not definably compact, then $N$ is infinite and $G/N$ is a definably compact group.  

In other words, every solvable definable group is a definable extension of a definably compact group by a definable torsion-free group.
\end{theo}

\begin{proof}
Because $N \subseteq G^0$ (\ref{torfree}$(a)$) and $G/G^0$ is finite (so definably compact), we can suppose $G$ definably connected.
We proceed by induction on $n = \dim G$. The case $n = 1$ is obvious, so let $\dim G = n > 1$.\\

If $G$ is \underline{abelian}, the theorem can be extracted from \cite[2.6]{ps05}. The argument is that if $G/N$ is not definably compact, then by \cite{ps} there is a definable 1-dimensional torsion-free subgroup $H$ in $G/N$,
and its preimage in $G$ would be a definable torsion-free subgroup of $G$, in contradiction with the maximality of $N$. \\

Let $G$ be now \underline{non-abelian}.
Since $G$ is solvable and definably connected there is a normal solvable definable subgroup $S < G$
such that $G/S$ is abelian and infinite. We distinguish the cases where
$S$ is definably compact and where $S$ is not definably compact.

\begin{itemize}

\item  If $S$ is definably compact then $G/S$ is not. By the abelian case
the maximal normal definable torsion-free subgroup $N_1$ of $G/S$ is infinite
and $(G/S)/N_1$ is definably compact. If $\pi \colon G \to G/S$ is the canonical projection, 
let $N^{\prime} = \pi^{-1}(N_1)$. By \ref{comp_tor} the definable exact sequence
\[
1\ \longrightarrow\ S \ \stackrel{i}{\longrightarrow}\ N^{\prime} \stackrel{\pi}{\longrightarrow}\ N_1\ \longrightarrow\ 1  
\]

\vs \noindent
splits definably in a direct product, thus $G$ contains a definable subgroup $N$ definably isomorphic to $N_1$ such that $N^{\prime} = S \times N$. Since $S \cong N^{\prime}/N$ which is definably compact, it follows that $N$ is the maximal normal definable torsion-free subgroup of $N^{\prime}$ (\ref{quocomp}), and  it is normal in $G$ as well (\ref{tormaxnor}). 

We claim that $G/N$ is definably compact. To prove it, it is enough to provide a normal
definable subgroup which is definably compact, such that quotienting by it
we obtain a definably compact group. One such a subgroup is $N^{\prime}/N$ which is definably isomorphic
to $S$, and the quotient  
$(G/N)/(N^{\prime}/N)$ is definably isomorphic to $(G/S)/N_1$, as the following diagram shows by ``the $3 \times 3$ lemma''.

\[
\xymatrix{
& & 1 \ar[d] & 1 \ar[d] & \\
1 \ar[r] & S \ar[r] & N^{\prime} \ar[r] \ar[d]& N_1 \ar[r] \ar[d] & 1\\
1 \ar[r] & S \ar[r] \ar@{<->}[u] & G \ar[r] \ar[d] & G/S \ar[r] \ar[d] & 1\\
& & G/N^{\prime} \ar@{<->}[r] \ar[d] & (G/S)/N_1 \ar[d] & \\
& & 1 & 1 & 
 }
\]

\vs
Thus $N$ is the maximal normal definable torsion-free subgroup of $G$ (\ref{quocomp}), it is infinite and the theorem is proved for the case where $S$ is definably compact.  \\


\item If $S$ is not definably compact, then by induction the maximal
normal definable torsion-free subgroup $N_1$ of $S$ is infinite
(possibly $N_1 = S$) and $S/N_1$ is definably compact. 
$N_1$ is normal in $G$ as we showed in \ref{tormaxnor}. 
If $G/N_1$ is definably compact then $N_1$ is the maximal normal definable torsion-free subgroup of $G$
and we are done. Otherwise, again by induction, its infinite maximal normal definable torsion-free subgroup
$N_2$ is such that $(G/N_1)/N_2$ is
definably compact. 

If $\pi \colon G \to G/N_1$ is the canonical
projection, then $N = \pi^{-1}(N_2) < G$ is torsion-free by
\ref{torfree} and $G/N$ is definably isomorphic to $(G/N_1)/N_2$ which is definably compact.
Again by \ref{quocomp} $N$ is the maximal normal definable torsion-free subgroup
of $G$ and it is infinite, since it contains $N_1$.
\end{itemize} 
\end{proof}

Edmundo proves in \cite{solv} that every solvable definable group $G$ has a normal definable subgroup $K_1 \times H_1$, where $H_1$ is torsion-free and $K_1$ is the maximal normal definably compact definably connected subgroup of $G$, such that $G/(K_1 \times H_1)$ is definably compact (\cite[5.8]{solv}). The theorem above shows that his $H_1$ corresponds to the maximal normal definable torsion-free subgroup of $G$.\\

\noindent
Several corollaries follow from Theorem \ref{solvquozcomp}:

%
%
\vs
\begin{cor}\label{maxtorsolvmax}
The maximal normal definable torsion-free subgroup of a solvable definable group contains any definable torsion-free subgroup of $G$, so it is the unique maximal definable torsion-free subgroup of $G$.
\end{cor}

\begin{proof}
It follows by \ref{solvquozcomp} and \ref{quocomp}.
\end{proof}

Since $N$ is torsion-free, we can see that $G/N$ is the maximal quotient which is definably compact:

\begin{cor}
Let $H$ be a normal definable subgroup of a solvable definable group $G$. Then $G/H$ is definably compact if and only if $H \supseteq N$, where $N$ is the maximal normal definable torsion-free subgroup of $G$.
\end{cor}

\begin{proof} \hspace{7cm}
\begin{enumerate}
\item[$(\Rightarrow)$] If not, $HN/H$ is an infinite definable torsion-free group, since it is definably isomorphic to $N/(H \cap N)$. This is in contradiction with the fact that $G/H$ is definably compact.   

\vspace{0.1cm}
\item[$(\Leftarrow)$] If $H \supseteq N$, then $G/H$ is definably isomorphic to the quotient of $G/N$ by $H/N$, so it is definably compact by \ref{solvquozcomp}.
\end{enumerate}
\end{proof}

\begin{cor}\label{compinsolvab}
Every definably compact subgroup of a solvable definably connected group is abelian.
\end{cor}

\begin{proof}
Let $K$ be a definably compact subgroup of a solvable definably connected group $G$. 
Let $N$ be the maximal normal definable torsion-free subgroup of $G$. Since $K \cap N = \{e\}$ (see \ref{duacomptf}), it follows that $K$ is definably isomorphic (by the canonical projection $\pi \colon G \to G/N$) to a definable subgroup of $G/N$ which is abelian by \cite[5.4]{ps00} (it is a solvable definably connected definably compact group).
\end{proof}

\begin{cor}\label{derinh}

The derived subgroup of a solvable definably connected group 
is contained in its maximal normal definable torsion-free subgroup.
\end{cor}

\begin{proof}
If $N$ is the maximal normal definable torsion-free subgroup of a solvable definably connected group $G$, then $G/N$ is abelian and so the derived subgroup $[G, G] \subseteq N$, because it is the smallest normal subgroup with abelian quotient.
\end{proof}

\begin{cor} \label{extcomprad}
Every definable group is a definable extension of a definable group with definably compact solvable radical by a definable torsion-free group. 
\end{cor}

\begin{proof}
Let $G$ be a definable group with solvable radical $R$ and maximal normal definable torsion-free subgroup $N$.
Then $G$ can be viewed as a definable extension of $G/N$ by $N$. By \ref{radgsun} $R/N$ is the solvable radical of $G/N$, and by \ref{solvquozcomp} $R/N$ is definably compact.
\end{proof}
\subsection{A Lie-like decomposition of solvable groups}
 
In analogy with finite groups, Strzebonski develops in \cite{str} a theory of definable $p$-groups, proving corresponding ``Sylow's theorems'', where the cardinality of finite $p$-groups is replaced by the \ominimal Euler characteristic of definable $p$-groups.
In addition he introduces $0$-groups, which are further investigated by Berarducci in \cite{zero}.

After reviewing facts about $0$-groups, we will show that every solvable definably connected group can be
decomposed as the product of its maximal normal definable torsion-free subgroup and any of its maximal $0$-subgroups (\ref{prodsolv}).

As a consequence we get that the maximal normal definable torsion-free subgroup of a solvable definably connected group always has cofactors,
and it has definable cofactors if and only if every $0$-subgroup is definably compact (\ref{solvsempr}).

\begin{dfn}(\cite{str}).  \label{dfn0gr}

\begin{enumerate} 
\item[-] A \textbf{$0$-group} is a definable group $G$ such that for every proper definable subgroup $H$, $E(G/H) = 0$ (in particular, $E(G) = 0$).  

\item[-] A \textbf{$0$-subgroup} is a definable subgroup which is a $0$-group.

\item[-] A \textbf{$0$-Sylow} is a maximal $0$-subgroup.


\end{enumerate}

\end{dfn}

\begin{fact} \emph{(\cite[2.9, 5.17, 2.21]{str}).}  \label{maxSylconj}  

\begin{enumerate}

\item[$(a)$] Every definable group $G$ with $E(G) = 0$ contains a $0$-subgroup.

\item[$(b)$] Every $0$-group is abelian and definably connected. 

\item[$(c)$] Every $0$-subgroup is contained in a $0$-Sylow.

\item[$(d)$] Any two $0$-Sylow are conjugate.

\item[$(e)$] If $H$ is a $0$-subgroup of a definable group $G$, then

\[
H \mbox{ is a $0$-Sylow } \Leftrightarrow\ E(G/H) \neq 0.
\]

\end{enumerate}
\end{fact}

\vs
\begin{dfn} (\cite[5.1]{str}) 
A group $G$ is said to be a \textbf{definable torus} if it is a $0$-group and every definably connected subgroup of $G$ is a 0-group.
\end{dfn}

\noindent
First Strzebonski (\cite[5.3]{str}) and then Peterzil and Steinhorn (\cite[5.6]{ps}) provided examples of $0$-groups which are not definable tori, and no one of them is definably compact. This is not accidental:

\begin{fact} \emph{(\cite[5.9]{zero}).} \label{torusis0gr}
A definable group is a definable torus if and only if it is abelian, definably connected and definably compact.
\end{fact}

\noindent
In Berarducci's paper the theorem above is stated for groups defined in an \ominimal expansion of a real closed field. Nevertheless the same proof holds in every \ominimal structure, once one has obtained that $E(G) = 0$ for every $G$ definably connected and definably compact (which follows from \ref{eu}, \ref{euconn} and \ref{compnoncomp}$(b)$).

\begin{lem}\label{charsyl}
Let $G$ be a definable group with $E(G) = 0$. If $H < G$ is a definable subgroup such that $E(G/H) \neq 0$, then
every $0$-Sylow of $H$ is a $0$-Sylow of $G$.
\end{lem}

\begin{proof}
Let $A$ be a $0$-Sylow of $H$. Then $E(G/A) = E(G/H)E(H/A) \neq 0$ and by \ref{maxSylconj}$(e)$ $A$ is a $0$-Sylow of $G$ as well.
\end{proof}

\begin{lem}\label{chara0syl}
Let $H$ be a definably connected subgroup of a definable group $G$ with $E(G) = 0$.
Then $H$ is a $0$-Sylow of $G$ if and only if $H$ is
of minimal dimension among the definable subgroups $P$ of $G$ such that $E(G/P) \neq 0$. 
\end{lem}

\begin{proof}
\hspace{7cm}
\begin{enumerate}
\item[$(\Rightarrow)$]
Suppose $P$ is a definable subgroup of $G$ such that $E(G/P) \neq 0$ and $\dim P < \dim H$. 
Because $E(G) = 0$ and $E(G/P) \neq 0$, $E(P) = 0$.
Let $A$ be a $0$-Sylow
of $P$. Then $A$ is a $0$-Sylow of $G$ (\ref{charsyl}) with $\dim A < \dim H$, in contradiction with the fact that any two $0$-Sylows are conjugate (\ref{maxSylconj}$(d)$).\\

\item[$(\Leftarrow)$]
Let $H$ be a definably connected subgroup of $G$ of minimal dimension among the definable subgroups $P$ of $G$ such that $E(G/P) \neq~0$. To see that $H$ is a $0$-Sylow, it is enough to check that it is a $0$-group (\ref{maxSylconj}$(e)$). 

For every definable proper $P < H$, $E(G/H)E(H/P) = E(G/P) = 0$, because of minimality and connectedness of $H$. Since $E(G/H) \neq 0$, it follows that $E(H/P) = 0$ and thus $H$ is a $0$-group.
\end{enumerate}
\end{proof}

We can show now the existence of a definable Lie-like decomposition (see \ref{complie}) for solvable definably connected groups, where the role of a maximal torus is played by a maximal $0$-subgroup (which is a maximal definable torus whenever it is definably compact: see \ref{torusis0gr}). 
 
\begin{theo}\label{prodsolv}
Let $G$ be a solvable definably connected group and let $N$ be its maximal normal definable torsion-free subgroup. Then for every $0$-Sylow $A$ of $G$
\[
G = AN.
\]

Moreover if $A_1 < G$ is a $0$-subgroup such that $G = A_1N$, then $A_1$ is a $0$-Sylow as well $($and therefore a conjugate of $A)$.
\end{theo}

\begin{proof}

If $G$ is torsion-free then $G = N$ and there is nothing to prove. So suppose $E(G) = 0$ (\ref{euconn}) and let $A$ be a $0$-Sylow of $G$.
We want to show that $G = AN$. Consider the following diagram:

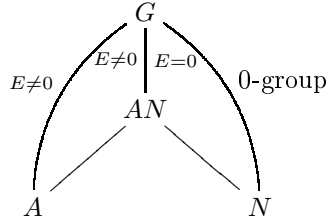
\[
\xymatrix{
& G \ar@/_1pc/@{-}[ddl]_{E \neq 0} \ar@{-}[d]^{E = 0}_{E \neq 0} \ar@/^1pc/@{-}[ddr]^{\mbox{0-group} } \\
& AN \ar@{-}[dl] \ar@{-}[dr] \\
A & & N
}
\]

\vs \noindent
Since $A$ is a $0$-Sylow, it follows that $E(G/A) \neq 0$ (\ref{maxSylconj}) and thus $E(G/AN) \neq 0$ as well, since $E(G/AN)E(AN/A) = E(G/A)$.  

On the other hand, $AN/N$ is a definable subgroup of the abelian, definably connected, definably compact group $G/N$ (\ref{solvquozcomp}) which is in particular a $0$-group (\ref{torusis0gr}), so $E(G/AN) = E((G/N)/(AN/N)) = 0$, unless $G/N = AN/N$. Therefore $G = AN$.

Let $A_1$ be a $0$-subgroup of $G$ such that $G = A_1N$, and let $S_1$ be a $0$-Sylow containing $A_1$. Thus  $S_1/A_1$ is definably isomorphic to $(S_1 \cap N)/(A_1 \cap N)$, and then $E(S_1/A_1) = \pm 1$ (\ref{torfree}$(d)$ and \ref{eu}). But $S_1$ is a $0$-group, so $S_1 = A_1$. 
\end{proof}

\begin{cor}\label{abconndec}
Let $G$ be an abelian definably connected group. Then there are unique definable subgroups $N$ and $A$ with $N$ torsion-free and $A$ $0$-group such that $G = AN$. $A$ is the only $0$-Sylow of $G$ and $N$ is the maximal definable torsion-free subgroup of $G$. 
\end{cor}
\begin{rem} \label{factor}
If $A$ is an abelian definably connected group and $H < A$ is a torsion-free definable subgroup, then they are both divisible (\cite{str}) and therefore $A$ contains a (possibly non-definable) direct cofactor of $H$ (see for instance \cite[10.24]{rot}). Thus
\[
A\ \cong\ H\ \times\ A/H.
\]

\vs
In particular this is the case whenever $A$ is a $0$-Sylow of a definable group $G$ (not necessarily solvable)
and $H$ is the maximal definable torsion-free subgroup of $A$.
\end{rem}

\begin{prop}\label{solvsempr}
Let $G$ be a solvable definably connected group and let $N$ be its maximal normal definable torsion-free subgroup. Then 
\begin{enumerate}
\item[$(a)$] $G$ is abstractly isomorphic to a semidirect product $N \rtimes G/N$.

\vspace{0.1cm}
\item[$(b)$] $N$ has definable cofactors in $G$  
if and only if every $0$-subgroup of $G$ is definably compact.
\end{enumerate}
\end{prop}

\begin{proof}
Let $A$ be a $0$-Sylow of $G$. Then by \ref{prodsolv} $G = AN$.

If $A$ is definably compact, then $A \cap N = \{e\}$ and $A$ is a definable cofactor of $N$ in $G$.  
Otherwise, $A \cap N$ is the maximal definable torsion-free subgroup of $A$ (because $N$ contains every torsion-free definable subgroup of $G$: \ref{quocomp} and \ref{solvquozcomp}) and it is infinite. 
If $T$ is a cofactor of $A \cap N$ in $A$ (see \ref{factor}), i.e. $A = (A \cap N) \times T$, then by \ref{prodsolv}
$T$ is a cofactor of $N$ in $G$. 

If $T$ were definable, then $E(A/T) = E(A \cap N) = \pm 1$, in contradiction with the fact that $A$ is a $0$-group. Then $T$ (and any other cofactor of $A \cap N$ in $A$) cannot be definable.

Suppose now by contradiction that $H$ is a definable cofactor of $N$ in $G$. Hence $H$ is a definable torus (it is definably isomorphic to $G/N$), so a $0$-group. Since $G = HN$,
by \ref{prodsolv} $H$ is a $0$-Sylow, so it is a conjugate of $A$ which is not definably compact, contradiction. 
\end{proof}
 
\vs
\section{A definable Iwasawa-like decomposition}

In this section we find an \ominimal analogue  
of the following decomposition of semisimple Lie groups:

\begin{fact}\label{iwasublie}
\emph{ (Iwasawa decomposition of semisimple Lie algebras and Lie groups) (\cite{iwa}, \cite[Chapter 6]{kna}).} 
For every semisimple Lie algebra \g\ over $\mathbb{C}$ there esists a basis $\{X_i\}$ of \g\ and 
subspaces $\mathfrak{k}, \la, \n$  
such that $\g$ is a direct sum $\g = \mathfrak{k} \oplus \la \oplus \n$, and
the matrices representing $\ad(\g)$ with respect to $\{X_i\}$ have the following properties:

\begin{enumerate}
\item[-] the matrices of $\ad(\mathfrak{k})$ are skew-symmetric,

\item[-] the matrices of $\ad(\la)$ are diagonal with real entries,

\item[-] the matrices of $\ad(\n)$ are upper triangular with $0$'s on the diagonal.
\end{enumerate}

\noindent
Let $G$ be a semisimple connected Lie group with finite center and Lie algebra $\g = \mathfrak{k} \oplus \la \oplus \n$ as above. If $K$, $A$ and $N$ are connected analytic subgroups of $G$ with Lie algebras $\mathfrak{k}$, \la\ and \n\ respectively, then: 

\begin{enumerate}
\item[$(a)$]the multiplication map

\begin{align*}
K \times A \times N\ &\longrightarrow\ G \\
(\, k\ ,\ a\ ,\ n\, )\ &\ \mapsto\ kan
\end{align*}

\noindent
is a surjective diffeomorphism;

\vspace{0.1cm}
\item[$(b)$] the group $K$ is a maximal compact subgroup of $G$ and any maximal compact subgroup of $G$ is a conjugate of $K$;

\vspace{0.1cm}
\item[$(c)$] $AN = NA$, i.e. $AN$ is a subgroup.
\end{enumerate}
\end{fact}


We recall now briefly some theory of Lie algebras for definable groups developed by Peterzil,
Pillay and Starchenko:

\begin{rem} (see \cite{pps1} and \cite{pps02}). \label{linliealg}
If $G$ is a definable group in an \ominimal expansion of a real closed field \Rs, denoted by $a(g) \colon G \to G$ the conjugation map given by $x \mapsto gxg^{-1}$, 
we can consider: $(1)$ $T_eG$, the tangent space of $G$ in the identity $e$, $(2)$ the \textbf{adjoint map}:  

\begin{align*}
\Ad \colon &G\ \longrightarrow\ \GL(T_eG)\\
&g\quad \mapsto\quad d_e(a(g))
\end{align*}

\vs \noindent
and $(3)$ the map $\ad \colon T_eG  \to \End(T_eG)$, $\ad = d_e(\Ad)$, where $d_e$ denotes the differential at the identity. 

The vector space $T_eG$ with the binary operation given by $[\xi, \zeta] = \ad(\xi)(\zeta)$ is a Lie algebra, it is definable, and  
it is called the \textbf{Lie algebra of $G$}.  

If we denote it by $\g = \Lie(G)$, then $\End_{\Rs}(\g)$, $\Aut_{\Rs}(\g)$ and the maps $\Ad \colon G \to \Aut_{\Rs}(\g)$, $\ad \colon \g \to \End_{\Rs}(\g)$ are all definable. 
Fixing a basis of $\g$, we may assume that $\g = \Rs^m$ and $\Ad \colon G \to \GL_m(\Rs)$, for some $m \in \N$.  
 
In particular, 
when $G = \GL_n(\Rs)$, we can identify $\Lie(G)$ with $M_n(\Rs)$, the set of all matrices $n \times n$, with bracket $[A, B] = AB - BA = \ad(A)(B)$, for every $A, B \in M_n(\Rs)$. With this convention
the adjoint map $\Ad \colon \GL_n(\Rs) \to \GL_{n^2}(\Rs)$ is given by $\Ad(A)(X) = A^{-1}XA$. 
\end{rem}

\subsection{Definably simple groups} Peterzil, Pillay and Starchenko deeply analyse definably simple groups
in \cite{pps1, pps2, pps02}.  
The following theorem (which builds on their papers) clarifies the structure
of a definably simple group in terms of definably compact and definable torsion-free subgroups, and it will be useful later on to understand the structure of any definable group.

\begin{nota}
Let $m \in \N$.

\begin{enumerate}
\item[-] $O_{m}(\Rs) = \{[x_{ij}] \in GL_m(\Rs) : [x_{ij}][x_{ji}] = I\}$ is the orthogonal group,

\vspace{0.1cm}
\item[-]
$T_m^+(\Rs) = \{ [x_{ij}] \in GL_m(\Rs) : x_{ij} = 0\ \forall \, i < j \mbox{ and } x_{ii} > 0\ \forall \, i\}$ is the group of upper triangular matrices with positive elements on the diagonal,

\vspace{0.1cm}
\item[-]
$UT_{m}(\Rs) = \{ [x_{ij}] \in GL_m(\Rs) : x_{ij} = 0\ \forall \, i < j \mbox{ and } x_{ii} = 1\ \forall \, i\}$ is the group of unipotent upper triangular matrices,

\vspace{0.1cm}
\item[-]
$D^+_{m}(\Rs) = \{[x_{ij}] \in GL_m(\Rs) : x_{ij} = 0\ \forall \, i \neq j,\ x_{ii} > 0\ \forall \, i\}$ is the group of diagonal matrices with positive elements on the diagonal.

\end{enumerate}

\end{nota}
 
\vs
\begin{theo}\label{declin}
Let $G$ be a definably simple group. Then there is a definable real closed field \Rs\ and some $m \in \N$ such that $G$ is definably isomorphic to a group $G_1 < \GL_{m}(\Rs)$ definable in \Rs, with the following properties:

\begin{enumerate}
\vspace{0.1cm}
\item[-] $G_1 = KH$, with $K = G_1 \cap O_{m}(\Rs)$ and $H = G_1 \cap T^+_{m}(\Rs)$,  
 
\vspace{0.1cm}
\item[-] $H = AN$, with $A = G_1 \cap D^+_{m}(\Rs)$ and $N = G_1 \cap UT_{m}(\Rs)$.

\end{enumerate}
\end{theo}

\begin{proof}
By \cite[4.1]{pps1}, there is a definable real closed field $\Rs$,  
such that we can suppose $G$ definable in $\Rs$  
and contained in $\GL_{n}(\Rs)$, for some $n \in \N$. 

Let $\g$ be the Lie algebra of $G$. By Theorem 2.36 of \cite{pps1}, \g\ is a simple Lie algebra over \Rs. As noticed in the proof of Theorem 5.1 in \cite{pps02}, there is a first order formula which says that there are  finitely many simple Lie subalgebras $\g_1, \dots, \g_r$ of $M_n(\Rs)$, such that any simple subalgebra of $M_n(\Rs)$ is isomorphic to one of the $\g_i$ (we suppose to know $r$). In addition we can require (in the same formula) that for every $i = 1, \dots, r$ there are subspaces $\mathfrak{k_i}, \la_i, \n_i$ of $\g_i$, with $\g_i = \mathfrak{k_i} \oplus \la_i \oplus \n_i$, such that the matrices representing $\ad(\g_i)$ have the following
properties:

\begin{enumerate}
\item[-] the matrices of $\ad(\mathfrak{k_i})$ are skew-symmetric,

\vspace{0.1cm}
\item[-] the matrices of $\ad(\la_i)$ are diagonal,

\vspace{0.1cm}
\item[-] the matrices of $\ad(\n_i)$ are upper triangular with $0$'s on the diagonal.
\end{enumerate}

All these properties are first order (see \ref{linliealg}). By the Iwasawa decomposition of semisimple Lie algebras (\ref{iwasublie}), this formula is true in $\R$ (with an abuse of notation we denote by $\R$ both
the set and the structure of ordered field $\langle \R, <, +, \cdot \rangle$ on it) 
and therefore it is true in \Rs\ as well. Hence
\g\ is isomorphic to a Lie algebra $\g_i \in \{\g_1, \dots, \g_r\}$ 
with the above mentioned properties. Say $\g_i = \g_1$.

By the proof of Theorem 2.37 in \cite{pps1}, $G$ is definably isomorphic to $\Aut_{\Rs}(\g)^0$, the definably connected component of the identity in $\Aut_{\Rs}(\g)$. Therefore $G$ is definably isomorphic also to $\Aut_{\Rs}(\g_1)^0 = G_1(\Rs)$, a semialgebraic linear group defined over $\Rs_{alg}$. So it makes sense to consider the group $G_1(\R)$ defined in $\R$
by the same formula over $\R_{alg}$ defining $G_1(\Rs)$ in $\Rs$.   
It is a simple Lie group equal to $\Aut_{\R}(\g_1)^0$.  
If $K$, $A$, $N$ are connected subgroups of $G_1(\R)$ corresponding to the Lie subalgebras $\mathfrak{k_1}$, $\la_1$ and $\n_1$, then by \ref{iwasublie} $AN = NA$ and $G_1(\R) = KAN$. Moreover if $m = n^2$,

\begin{enumerate}
\item[-] the matrices of $\ad(\mathfrak{k_1})$ are skew-symmetric $ \Rightarrow\ K \subseteq O_{m}(\R)$, 

\vspace{0.1cm}
\item[-] the matrices of $\ad(\la_1)$ are diagonal $ \Rightarrow\ A \subseteq D^+_{m}(\R)$,

\vspace{0.1cm}
\item[-] the matrices of $\ad(\n_1)$ are upper triangular with $0$'s on the diagonal $ \Rightarrow\ N \subseteq UT_{m}(\R)$.
\end{enumerate}

\vs \noindent
Since $D^+_{m}(\R) \cap UT_{m}(\R) = \{I\} = O_{m}(\R) \cap T^+_{m}(\R)$, we get that

\vspace{0.1cm}
\begin{enumerate}
\item[-] $K = G_1(\R) \cap O_{m}(\R)$,

\vspace{0.1cm}
\item[-] $A = G_1(\R) \cap D^+_{m}(\R)$,

\vspace{0.1cm}
\item[-] $N = G_1(\R) \cap UT_{m}(\R)$.
\end{enumerate}

Thus the first order formula in the language of ordered fields which says that every element $g \in G_1$ can be written in a unique way as a product $g = kan$, with $k \in G_1 \cap O_{m}$, $a \in G_1 \cap D^+_{m}$, $n \in G_1 \cap UT_{m}$ and $an = na\ \forall\, a \in G_1 \cap D^+_{m}, \forall\, n \in G_1 \cap UT_{m}$ is true in $\R$  
and therefore it is true in $\Rs$  
as well.
\end{proof}

\begin{cor} \label{conjmax}
Any definably simple group has maximal definably compact subgroups, all definably connected and conjugate.
Moreover every such a maximal definably compact subgroup $K$ has a definable torsion-free complement $H$.
\end{cor}

\begin{proof}
Let $G_1 < \GL_m(\Rs)$ be a definable group definably isomorphic to $G$, as in Theorem \ref{declin}.
As we noticed in \ref{duacomptf}, $K = G_1 \cap O_m(\Rs)$ is a maximal definably compact subgroup of $G_1$.

If $C$ is any definably compact subgroup of $G_1$, we want to show that $C$ is contained in a conjugate of $K$.
Since every definably compact subgroup of $\GL_m(\Rs)$ is semialgebraic (\cite[4.6]{pps02}),
the fact that $C$ is a definably compact (i.e. closed and bounded by \cite{ps}) definable subgroup of $\GL_m(\Rs)$
can be expressed by a first order formula in the language of ordered fields. 
Suppose that now $C = C(y)$ is defined over a set of parameters $y = (y_1, \dots, y_n)$. 
Since every compact (again closed and bounded) subgroup of $G_1(\R)$ is contained in a conjugate of $K(\R) = G_1(\R) \cap O_m(\R)$, the following formula  
\[
\forall y\ [C(y) \mbox{ is a closed and bounded subgroup of } G_1 \Rightarrow \exists x \in G_1\  (C(y) \subset K^x)] 
\]

\vs \noindent
is true in $\R$,  
so it is true in $\Rs$ as well.

Because $G$ and $H$ are definably connected, also $K$ (and every conjugate of it) is definably connected. 

Note that $G = K^xH^x$ for every $x \in G$, and then $H^x$ is a definable torsion-free complement of $K^x$.
\end{proof}

\subsection{Groups with definably compact solvable radical} Using results of section $3$ and $4$ together with \ref{declin} and \ref{conjmax}, we can prove
an analogous definable decomposition for definably connected groups with definably compact solvable radical:

\begin{prop}\label{maxcomprad}
Let $G$ be a definably connected group and let $R$ be its solvable radical. If $R$ is definably compact then 

\begin{enumerate}
\vspace{0.1cm}
\item $G$ has maximal definably compact subgroups and they are all definably connected and conjugate to each other.

\vspace{0.1cm}
\item Every such a maximal definably compact subgroup $K$ has a definable torsion-free complement $H < G$ $($which is a maximal definable torsion-free subgroup of $G)$ and $H = AN$, where $A$ is abelian and $N$ is nilpotent, and both are definable subgroups. 
\end{enumerate}
\end{prop}

\begin{proof}
We suppose that the definably connected semisimple group $G/R$ is centerless. Even if it is not, the center of $G/R$ is finite, so its preimage in $G$ is a normal solvable definably compact subgroup, which is the only assumption we will make on $R$.

By \cite[4.1]{pps1}, $G/R$ is a direct product of definably simple groups $G_i$, $i = 1, \dots, s$. By the definably simple case (\ref{declin}), for every $i = 1, \dots, s$, $G_i = K_iH_i$, where $K_i$ is definably compact and $H_i$ is torsion-free, and every definably compact subgroup of $G_i$ is contained in a conjugate (in $G_i$) of $K_i$ (\ref{conjmax}). Let $\bar{K} = \{(k_1, \dots, k_s) : k_i \in K_i\}$ and $\bar{H} =\{(h_i, \ldots, h_s) : h_i \in H_i\}$ in $G/R$. Again by the simple case, for every $i = 1, \dots, s$, $H_i = A_iN_i$, with $A_i$ abelian and $N_i$ nilpotent definable subgroups. Thus $\bar{H} = \bar{A}\bar{N}$, with $\bar{A}$ the product of the $A_i$'s and $\bar{N}$ the product of the $N_i$'s.
Note that $\bar{K}$ is definably compact, $\bar{H}$ is torsion-free and $G/R = \bar{K}\bar{H}$.

Let $\pi \colon G \to G/R$ be the canonical projection. Since $R$ is definably compact, $K = \pi^{-1}(\bar{K}) \subseteq G$ is definably compact. Define $\bar{H}_1 = \pi^{-1}(\bar{H}) \subseteq G$. 
The group $\bar{H}_1$ is a definable extension of the definable torsion-free group $\bar{H}$ by the solvable definably compact group $R$. Hence by \ref{comp_tor} the definable exact sequence

\[
1\ \longrightarrow\ R\ \stackrel{i}{\longrightarrow}\ \bar{H}_1\ \stackrel{\pi}{\longrightarrow}\ \bar{H}\ \longrightarrow\ 1  \\
\]

\vs \noindent
splits definably in a direct product $\bar{H}_1 = R \times H$, for some definable torsion-free subgroup $H < G$
definably isomorphic to $\bar{H}$. It is immediate that $G = KH$. 

As we noticed in \ref{duacomptf}, $K$ is a maximal definably compact subgroup of $G$ and $H$ is a maximal definable torsion-free subgroup of $G$. Therefore for every $g \in G$, $K^g$ is a maximal definably compact subgroup of $G = K^gH^g$. We want to show that they are the only ones,
which is equivalent to say that for every definably compact subgroup $C$ of $G$, there is some $g \in G$ such that $C \subset K^g$.  

For every $i = 1, \dots, s$, let $\pi_i \colon G \to G_i$ be the composition of the canonical projections, first on $G/R$ and then from $G/R$ to $G_i$. By the simple case (\ref{declin}) we know that $\pi_i(C) \subset K_i^{g_i}$ for some $g_i \in G_i$ and so $\pi(C) \subset \bar{K}^{\bar{g}}$ with $\bar{g} = (g_1, \dots, g_s) \in G/R$. Therefore $C \subset K^g$, for every $g \in \pi^{-1}(\bar{g})$.

Finally, because $G$ and $H$ are definably connected, also $K$ (and every conjugate of it) is definably connected. 
\end{proof}

\begin{ex}
Let \Rs\ be a real closed field. Every element of the semisimple definable group $G = \SL_2(\Rs) = \{A \in \GL_2(\Rs) : \det A = 1 \}$  
can be written as a product

\begin{displaymath}
\left( \begin{array}{cc}
a & -b  \\
b & a   
\end{array} \right) 
\left( \begin{array}{cc}
c & 0  \\
0 & 1/c   
\end{array} \right) 
\left( \begin{array}{cc}
1 & d  \\
0 & 1   
\end{array} \right), 
\end{displaymath} 

\vs \noindent
for unique $a,\ b,\ c,\ d \in \Rs$ such that $a^2 + b^2 = 1$ and $c > 0$.

\begin{align*}
\left(\ \begin{matrix} 
a & -b  \\
b & a   
\end{matrix}\ \right)\ \, & \in\  \SO_2(\Rs) = G \cap O_2(\Rs) = K, \\
 \left(\ \begin{matrix}  
c & 0  \\
0 & 1/c   
\end{matrix}\ \right)\  & \in\  G \cap D^+_2(\Rs) = A,\\ 
 \left(\  \begin{matrix} 
1 & d  \\
0 & 1   
\end{matrix}\  \right) \ \: \: & \in \: \:  G \cap  UT_2(\Rs) = N. \\
\end{align*}


%
\end{ex}


\vs
\section{The maximal quotient with definably compact solvable radical}

In this section we point out some properties of the projection $\pi \colon G \to G/N$ by the maximal normal definable torsion-free subgroup $N$ of a definable group $G$. 

First we show that the group $G/N$ is the maximal quotient of $G$ with definably compact solvable radical: 

\begin{lem}
Let $N$ be the maximal normal definable torsion-free subgroup of a definable group $G$. If $H$ is a normal definable subgroup of $G$ such that the solvable radical of $G/H$ is definably compact, then $N \subseteq H$.
\end{lem}

\begin{proof}
Since $N$ is contained in the solvable radical of $G$, it follows that $NH/H$ is contained in the definably compact solvable radical $\bar{R}$ of $G/H$. But $NH/H$ is definably isomorphic to $N/(N \cap H)$ which is torsion-free (\ref{torfree}$(d)$), so $NH = H$ and $N \subseteq H$.
\end{proof}
\begin{cor}
If $N$ is the maximal normal definable torsion-free subgroup of a definable group $G$, then $G/N$ is the maximal quotient of $G$ with definably compact solvable radical.
\end{cor}

\begin{proof}
The radical of $G/N$ is definably compact by \ref{radgsun} and \ref{solvquozcomp}. Maximality follows by our previous lemma. 
\end{proof}

Applying \ref{maxcomprad} to $G/N$, we get the following theorem:


\begin{theo}\label{Iwagisuacca}
Let $G$ be a definably connected group and $N$ its maximal normal definable torsion-free subgroup. Then the definable group $G/N$ has maximal definably compact subgroups and they are all definably connected and conjugate to each other. Moreover, every such a maximal definably compact subgroup $K$ has a definable torsion-free complement $H$ in  $G/N$,
i.e. 

\[
G/N = KH \qquad \mbox{and} \qquad K \cap H = \{e\}.
\]

\end{theo}

\vs
\begin{cor}\label{maxtorfreenon}
Every definable group has maximal definable torsion-free subgroups. 
\end{cor}

\begin{proof}
Let $N$ be the maximal normal definable torsion-free subgroup of a definable group $G$ and let $\pi \colon G \to G/N$ be the canonical projection. By the theorem above
$G^0/N = KH$, where $K$ is a definably compact subgroup and $H$ is a definable torsion-free subgroup.  
We claim that $\pi^{-1}(H)$ is a maximal definable torsion-free subgroup of $G$. 

If not, let $\bar{H} \supsetneq \pi^{-1}(H)$ be a definable torsion-free subgroup of $G$. Since definable torsion-free groups are definably connected, it follows that $\dim \bar{H} > \dim \pi^{-1}(H) = \dim N + \dim H$. So by dimension issues, $K \cap \pi(\bar{H}) \neq \{e\}$, contradiction (see \ref{duacomptf}).

\end{proof}
\begin{rem}
Since every $0$-group is abelian and definably connected (\cite{str}), and every definable subgroup of a definably compact group is definably compact, it follows that every $0$-subgroup of a definably compact group is a definable torus (\cite{zero}, see \ref{torusis0gr}). 

Therefore a definable subgroup of a definably compact group is a $0$-Sylow if and only if it is a maximal definable torus. 
\end{rem}

\begin{lem} \label{0subctf}
Let $G$ be a definable group such that
\[
G^0 = KH,
\]

\noindent
for some definably compact subgroup $K$ and some definable torsion-free subgroup $H$. Then every $0$-subgroup of $G$ is definably compact $($hence a definable torus$)$.
\end{lem}

\begin{proof}
Let $T$ be a maximal definable torus of $K$. Then by \ref{euprodsub}, \ref{eu} and \ref{maxSylconj}$(e)$

\[
E(G/T) = E(G/G^0)E(G^0/T) = E(G/G^0) E(K/T) E(H) \neq 0,
\]

\vs \noindent
and again by \ref{maxSylconj}$(e)$ $T$ is a $0$-Sylow of $G$.
\end{proof}

\begin{cor} \label{0subgn}
Let $G$ be a definable group and $N$ its maximal normal definable torsion-free subgroup. Then all $0$-subgroups
of $G/N$ are definably compact.
\end{cor}

\begin{proof}
See \ref{Iwagisuacca} and \ref{0subctf}.
\end{proof}
\begin{cor} \label{maxtfA}
Let $G$ be a definable group, $N$ its maximal normal definable torsion-free subgroup, and $A$ a $0$-Sylow of $G$. Then $A \cap N$ is the maximal $($normal$)$ definable torsion-free subgroup of $A$.
\end{cor}

\begin{proof}
Since $A/(A \cap N)$ is definably isomorphic to $AN/N$ which is definably compact (\ref{0subgn}), it follows that $A \cap N$ is the maximal (normal) definable torsion-free subgroup of $A$ (\ref{quocomp}).
\end{proof}

\begin{rem}
If $H$ is a definable subgroup of $G$, then $H \cap N$ may be strictly contained in the maximal normal definable torsion-free subgroup of $H$. For instance, one can take a non-normal
maximal definable torsion-free subgroup $H$ of $G$ (see \ref{maxtorfreenon}).
\end{rem}

\begin{lem} \label{syltor}
Let $N$ be the maximal normal definable torsion-free subgroup of a definable group $G$ and $\pi \colon G \to G/N$ the canonical projection. Then 
\begin{enumerate}
\item[$(a)$] $A$ is a $0$-Sylow of $G\ \Rightarrow\ \pi(A)$ is a maximal torus of $G/N$;

\vspace{0.1cm}
\item[$(b)$] $T$ is a maximal torus of $G/N\ \Rightarrow\ $ there is a $0$-Sylow $A$ of $G$ such that $\pi(A) = T$.  
\end{enumerate}

\end{lem}

\begin{proof} \hspace{7cm}
\begin{enumerate}

\item[$(a)$]
Define $\bar{A} := \pi(A)$. By \ref{chara0syl}, $\bar{A}$ is a $0$-Sylow of $\bar{G}: = G/N$ if and only if it is of minimal dimension among the definable subgroup $\bar{H}$ of $\bar{G}$ such that $E(\bar{G}/\bar{H}) \neq 0$.

First note that 
\[
E(\bar{G}/\bar{A}) = E((G/N)/(AN/N)) = E(G/AN) \neq 0,
\]

\vs \noindent
since $E(G/A) \neq 0$ and $E(G/A) = E(G/AN) E(AN/A)$ (\ref{maxSylconj} and \ref{euprodsub}).

For a contradiction, suppose  $\bar{H}$ is a definable subgroup of $\bar{G}$ such that $\dim \bar{H} < \dim \bar{A}$ and $E(\bar{G}/\bar{H}) \neq 0$. 

Define $H : = \pi^{-1}(\bar{H})$. Then $\dim H = \dim \bar{H} + \dim N < \dim \bar{A} + \dim N = \dim A$. Then $E(G/H) = 0$ by \ref{chara0syl}. But also
\[
E(G/H) = E(G/HN) E(HN/H) = E(\bar{G}/\bar{H}) E(N/(H \cap N)) \neq 0,
\]

\noindent
contradiction.\\

\item[$(b)$]
Let $T$ be now a maximal torus of $G/N$ and define $G_1: = \pi^{-1}(T)$. 

Because $G_1$ is solvable, by \ref{prodsolv} $G_1 = A_1N$, for any $0$-Sylow $A_1$ of $G_1$.
Since 
\[
E(G/A_1) = E(G/G_1)E(G_1/A_1) = E(\bar{G}/T)E(G_1/A_1) \neq 0,
\]

\vs \noindent
it follows by \ref{charsyl} that $A_1$ is a $0$-Sylow of $G$, and $\pi(A_1) = T$.
\end{enumerate}   
\end{proof}

\vs
\section{Definable Levi subgroups}

It is well-known that connected Lie groups admit the following decomposition:

\begin{fact} \emph{(\cite[Cap 1, 4]{oni3}).}
Let $L$ be a connected Lie group. If $R$ is the solvable radical of $L$, then there is a connected semisimple subgroup $S$ of $L$ such that 
\[
L = RS \qquad \mbox{ and } \qquad \dim(R \cap S) = 0.
\]

\end{fact}

\vs \noindent
A subgroup $S < L$ with such a property is said to be a \textbf{Levi subgroup} of $L$ and the decomposition above is said to be a \textbf{Levi decomposition} of $L$. Levi subgroups are all conjugate to each other (see \cite[4.3 Cor.3]{oni3}).

A Levi subgroup needs not be a Lie subgroup, it may be not closed.  
But if $L$ is compact, then the derived subgroup $[L, L]$   
is a closed Levi subgroup of $L$ (\cite[9.24]{hm}).

Hrushovski, Peterzil and Pillay prove in \cite{hpp2} the corresponding result for the \ominimal context: 
\begin{fact} \emph{(\cite[6.4]{hpp2}).} \label{decomp}
Let $G$ be a definably connected definably compact group. Then the derived subgroup $[G, G]$ of $G$ is a semisimple definable subgroup and  

\[
G = Z(G)^0 [G, G].
\]

\end{fact}

\vspace{0.1cm}
\begin{rem}
In general the derived subgroup $[G, G]$ of a definable group $G$ is a countable union of definable sets, and it is not necessarily definable.
\end{rem}

\begin{dfn}
We say that a definably connected group $G$ has a \textbf{definable Levi decomposition} if $G$ contains a semisimple definably connected subgroup $S$ such that $G = RS$, where $R$ is the solvable radical of $G$. We call such an $S$ a \textbf{definable Levi subgroup} of $G$.
\end{dfn}

As we noticed in \ref{semisnosolv}, semisimple definable groups have no infinite solvable normal definable subgroup. Therefore if $R$ and $S$ are as above, then $R \cap S$ is finite and $G$ is an almost semidirect product of $R$ and $S$. \\

Peterzil, Pillay and Starchenko prove in \cite{pps02} that definably connected linear groups have a definable Levi decomposition:

\begin{fact} \emph{(\cite[4.5]{pps02}).} \label{levilin}
Let $G$ be a definably connected linear group. Then $G$ is an almost semidirect product $G = RS$ of the solvable radical $R$ and a semisimple definably connected subgroup $S$.
\end{fact}

As a corollary of results in \cite{hpp2},
we can find a definable Levi decomposition also for definably connected groups $G$ such that the quotient by their center $Z(G)$ is definably compact:

\begin{prop} \emph{(\cite[6.4, 6.5]{hpp2}).} \label{leviquozcomp}
Let $G$ be a definably connected group such that $G/Z(G)$ is definably compact. If $R$ is the solvable radical of $G$, then $G = RS$ for some semisimple definably connected subgroup $S$ of $G$.
\end{prop}

\begin{proof}
We proceed by induction on $n = \dim G$. If $n = 1$ then $G$ is abelian, so suppose $n > 1$.

By \cite[6.4]{hpp2} $G/Z(G)$ has a definable Levi decomposition $G/Z(G) = R_1S_1$, where $R_1$ is the definably connected component of the identity in the center of $G/Z(G)$ and $S_1$ is the derived subgroup of $G/Z(G)$. Let $\pi \colon G \to G/Z(G)$ be the canonical projection and let
$H = \pi^{-1}(S_1)^0$. 

If $R_1$ is infinite (i.e. $G/Z(G)$ is not semisimple) then $\dim H < \dim G$ and by induction there exists a semisimple definably connected subgroup $S$ of $G$ such that $H = R_HS$, where $R_H$ is the solvable radical of $H$. Therefore $G = RS$.

If $G/Z(G)$ is semisimple then Hrushovski, Peterzil and Pillay prove in \cite[6.5]{hpp2} that $[G, G]$ is a semisimple definable subgroup of $G$ such that $G = Z(G)^0[G, G]$.
\end{proof}

In \ref{levigsuncomp} we find a definable Levi decomposition also when $G/N$ (or equivalently $G/R$) is definably compact.

The general problem to characterize definable groups with a definable Levi decomposition will be considered in another paper where
the proof that every definable group has $\bigvee$-definable (not always definable) ``Levi subgroups'' will be also given. 

We want to show now that (when they exist) definable Levi subgroups are all conjugate to each other. In the proof we consider several cases. For centerless groups we use the following theorem due to Peterzil, Pillay and Starchenko which allow us to reduce to the linear case:

\begin{fact} \cite[3.1, 3.2]{pps1} \label{centerl}
Let $G$ be a centerless definably connected group. Then there are definable real closed fields $\Rs_i$, $i = 1, \dots, s$, and definable groups $H_i < \GL(n_i, \Rs_i)$, $H_i$ definable in $\Rs_i$, such that 
$G$ is definably isomorphic to the direct product $H_1 \times \dots \times H_s$. 
\end{fact}

\begin{fact}\label{finsubincenter}
Every finite normal subgroup of a definably connected group is contained in its center.
\end{fact}

\begin{proof}
Let $G$ be a definably connected definable group and $F$ a normal finite subgroup. We want to show that $F \subseteq Z(G)$. For every $x \in F$, the image of the continuous map
\begin{align*}
G\ &\longrightarrow\ F \\
g\ &\mapsto\ gxg^{-1}
\end{align*}

\vs \noindent
is definably connected, so it has to be constant. Therefore $gxg^{-1} = x$ for every $g \in G$ and $x \in Z(G)$.
\end{proof}
We recall that a group $G$ is said to be \textbf{perfect} if $G = [G, G]$.

\begin{fact} \emph{(\cite{hpp2}).} \label{semisperf}
Every semisimple definably connected group is~perfect.
\end{fact}

\begin{rem} \label{decder}
If $H < G$ are groups such that $G = Z(G)H$, then $[G, G] = [H, H]$.
\end{rem}

\begin{proof}
Every $g \in G$ can be written as $g = zh$ with $z \in Z(G)$ and $h \in H$. Therefore for every $g_1, g_2 \in G$,

\[
[g_1, g_2] = [z_1h_1, z_2h_2] = [h_1, h_2].
\]
\end{proof}

\begin{prop} \label{leviconj}
If a definably connected group has a definable Levi decomposition, then all definable Levi subgroups are conjugate.
\end{prop}

\begin{proof}

Let $G$ be a definably connected group and let $R$ be its solvable radical.  
Suppose $G = RS$ is a definable Levi decomposition of $G$. We want to show that every definable Levi subgroup $\hat{S}$ of $G$ is a conjugate of $S$. Note that $\dim S = \dim \hat{S} = \dim G - \dim R$.

We proceed by induction on $n = \dim G$. If $n = 1$ then $G$ is abelian, so $G = R$.
Let $n > 1$. Our plan is to transfer from the field of the reals the property that Levi subgroups are all conjugate.  
We can do it directly in the semialgebraic linear case by results in \cite{pps1} and \cite{pps02}, then we can reduce from the non-semialgebraic to the semialgebraic case by \cite[4.1]{pps02}, and from the non-linear to the linear case by \cite[3.1, 3.2]{pps1} and induction. In detail:\\

$(1)$. Suppose first that $G$ is \textsc{linear}: $G < \GL_n(\Rs)$ for some definable real closed field $\Rs$ and some $n \in \N$. Then by \cite{pps1} and \cite{pps02}, $S$ and $\hat{S}$ are semialgebraic over $\Rs$ with parameters in $\Rs_{alg}$.\\

Let $G$ be \underline{semial}g\underline{ebraic}, say $G = G(x)$, with $x = (x_1, \dots, x_s)$ the parameters over which $G$ is defined.  
By \cite{pps1} and \cite{pps02} (see \ref{linliealg}), the formula  

\[
\forall\, x\ [\, S, \hat{S} < G(x)\  \Rightarrow\ \exists\, g \in G(x)\ (\hat{S} \subseteq S^g) \, ]
\]

\vs \noindent
is a first ordered sentence in the language of ordered fields and it is true in the field of the real numbers, therefore it is true in \Rs\ as well.\\ 

If $G$ is \underline{non-semial}g\underline{ebraic}, then by \cite[4.1]{pps02} there are definably connected semialgebraic $G_1 \lhd G_2 < \GL_n(\Rs)$ such 
that $G_1 \lhd G \lhd G_2$ and $G_2/G_1$ is abelian. 
By \cite{hpp2} semisimple definably connected group are perfect, so they have no abelian proper quotient.
Therefore $S$ and $\hat{S}$ are both contained in $G_1$ and we can reduce to the semialgebraic case. \\

$(2)$. Let $G$ be now \textsc{not necessarily linear}. We consider the possible cases for the center of $G$:\\

If $Z(G)$ is \underline{infinite} then $Z(G)^0 \subseteq R$ is infinite too. 

\begin{enumerate}
\item[$(a)$] Let $Z(G)^0 = R$. Then $S = [S, S] = [G, G] = [\hat{S}, \hat{S}] = \hat{S}$ (see \ref{semisperf} and \ref{decder}).\\

\item[$(b)$] Let $Z(G)^0 \subsetneq R$ and let $\pi \colon G \to G/Z(G)^0$ be the canonical projection. Then
$\pi(S)$ and $\pi(\hat{S})$ are definable Levi subgroups of $G/Z(G)^0$. By induction they are conjugate in $G/Z(G)^0$, so $Z(G)^0\hat{S} = Z(G)^0S^g = G_1$, for some $g \in G$. Since $\hat{S}$ and $S^g$ are definable Levi subgroups of $G_1$, by induction
it follows that they are conjugate in $G_1$, and therefore $\hat{S}$ and $S$ are conjugate in $G$.\\

\end{enumerate}

If $G$ is \underline{centerless} then we can reduce to the linear case by \cite[3.1, 3.2]{pps1} (\ref{centerl}) and \cite[4.5]{pps02} (\ref{levilin}).\\

Let $Z(G)$ be \underline{finite} but \underline{non-trivial} and let $\pi \colon G \to G/Z(G)$ be the canonical projection. 

Since every finite normal subgroup of $G$ is contained in its center (\ref{finsubincenter}), it follows that the center of $G/Z(G)$ cannot be finite and non-trivial. So we have to check the two cases where $\bar{G}:= G/Z(G)$ is centerless or with infinite center:\\

\begin{enumerate}
\item[-] If $\bar{G}$ is centerless then $\pi(S)$ and
$\pi(\hat{S})$ are conjugate in $\bar{G}$ by the linear case. Thus again $Z(G)\hat{S} = Z(G)S^g = G_1$ for some 
$g \in G$. It follows that $G_1^0 = \hat{S} = S^g$.\\

\item[-] If the center of $\bar{G}$ is infinite then $\pi(S)$ and $\pi(\hat{S})$ are conjugate in $\bar{G}$ for 
the case with infinite center. Again $Z(G)\hat{S} = Z(G)S^g = G_1$ for some 
$g \in G$ and $G_1^0 = \hat{S} = S^g$. 
\end{enumerate}
\end{proof}

\vs
\section{Definable groups of matrices}

The main result of this section is that every definably connected linear group $G$ can be decomposed as
$G = KH$, for some definably compact subgroup $K$ and some definable torsion-free subgroup $H$ (\ref{ctflin}). As we already observed in \ref{duacomptf}, since $K \cap H = \{e\}$, it follows that $K$ is a maximal definably compact subgroup and $H$
is a maximal definable torsion-free subgroup. \\

Definable linear groups are studied by Peterzil, Pillay and Starchenko in \cite{pps02}.
Among their results, together with \cite[1.41]{pps02} (see \ref{levilin}) we will use in an essential way the following:

\begin{fact} \emph{(\cite{pps02}).} \label{linab} 
Every abelian definably connected linear group is definably isomorphic to
\[
\SO_2(M)^k\ \times\ (\Rpl)^s\ \times (\Rx)^p
\]

\vs \noindent
for some integers $k, s, p \geq 0$, where \Rpl\ denotes the additive group $(M, +)$ and  \Rx\ the multiplicative group $(M^{>0}, \cdot)$. \\
\end{fact}

\begin{cor}\label{lin0}
Every linear $0$-group is definably compact.
\end{cor}

\begin{proof}
Let $G$ be a linear $0$-group.
By \ref{maxSylconj}$(b)$ and Fact \ref{linab}, there are non-negative integers $k, s, p$ such that $G$ is definably isomorphic to $\SO_2(M)^k\ \times\ (\Rpl)^s\ \times (\Rx)^p$. Since for every definable subgroup $H$ of $G$ we have $E(G/H) = 0$, it follows that $s = p = 0$ and $G$ is definably compact.
\end{proof}
\begin{prop} \label{lintorfreenoab}
Every linear definable torsion-free group is a product of definable $1$-dimensional subgroups.
\end{prop}

\begin{proof}
Let $\dim G = n$. We want to find 1-dimensional definable torsion-free subgroups $A_1, \ldots, A_n$ such that $G = A_1 \cdots A_n$. 

We proceed by induction on $n$. If $n = 1$ there is nothing to prove, so suppose $n > 1$. 

Let $G_1$ be a normal definable subgroup of $G$ such that $\dim G_1 = n - 1$ (\ref{torfree}$(c)$). We want to find a $1$-dimensional definable complement of $G_1$ and conclude using induction on $G_1$. 

Let $x \notin G_1$ and $Z:= Z(C_G(x))$, the center of the centralizer of $x$ in $G$. Being 
an abelian subgroup of $G$, $Z$ is by \ref{linab} a direct product of $1$-dimensional definable subgroups $Z_1, \dots, Z_s$. Since $x \in Z\ \backslash\ G_1$, it follows that $Z \not \subset G_1$ and so there is some $i \in \{1, \dots, s\}$ such that $Z_i \not \subset G_1$.  

The definable $1$-dimensional torsion-free group 
$Z_i$ does not contain any proper definable subgroup and therefore $Z_i \cap G_1 = \{e\}$. Moreover $\dim G - \dim G_1 = 1$ and $G$ is definably connected, thus $G = Z_iG_1$. By induction $G_1 = A_2 \cdots A_n$ and therefore $G = Z_iA_2 \cdots A_n$.
\end{proof}

\begin{rem}
Actually the proof of the proposition above shows the more general fact that if every abelian definable subgroup of a definable torsion-free group $G$ splits in the product of definable $1$-dimensional subgroups, so
does $G$.
\end{rem}

\begin{prop}\label{linsol}
Let $G$ be a solvable definably connected linear group 
and let $N$ be its maximal normal definable torsion-free subgroup. Then $G$ is definably isomorphic to a definable semidirect product $\SO_2(M)^k \ltimes N$. 
\end{prop}

\begin{proof}
If $G$ is not torsion-free, by \ref{prodsolv} $G = AN$ for every 0-Sylow $A$ of $G$. By \ref{linab} and \ref{lin0} $A$ is definably compact and definably isomorphic to $\SO_2(M)^{\dim A}$.
Therefore $A \cap N = \{e\}$ and the thesis follows.
\end{proof}

Proposition \ref{linsol} can be viewed as an \ominimal analogue of the fact that every solvable connected linearizable real Lie group can be decomposed into the semidirect product $T \ltimes N$ of a Lie subgroup $T$ isomorphic to $\SO_2(\R)^d$, for some $d \in \N$, and a contractible normal Lie subgroup $N$
(see for example \cite[7.1]{oni3}).\\

Generalizing Theorem \ref{solvsempr} about solvable groups,
we will see that any definably connected extension of a definably compact group by a definable torsion-free group splits abstractly (\ref{extcomp}). In the case of linear groups we prove now that every such an extension splits definably. In both proofs we 
will use  
the following lemma which gives us a sufficient condition to find a $0$-Sylow in the normalizer of a definable subgroup:

\begin{lem} \label{binorm}
Let $N < H < G$ be definable groups, with $N, H$ both normal in $G$, $N$ torsion-free, $E(G) = 0$.
Suppose that $N$ has a definable cofactor $K$ in $H$ such that every conjugate of $K$ in $G$ is a conjugate of $K$ in $H$. Then there is some $0$-Sylow of $G$ contained in $N_G(K)$, the normalizer of $K$ in $G$.
\end{lem}

\begin{proof} 
By hypothesis for every $g \in G$ there is some $x \in N$ such that $K^g = K^x$. This gives a well-defined bijective definable map  

\[
G/N_G(K)\ \longleftrightarrow\ N/N_N(K).
\]

\vs \noindent
Since $N$ is torsion-free, by \ref{torfree}$(d)$ and \ref{eu} it follows that $E(N/N_N(K)) = \pm 1 = E(G/N_G(K))$, and by \ref{charsyl} every $0$-Sylow of
$N_G(K)$ is a $0$-Sylow of $G$ as well.
\end{proof}

\begin{prop}\label{linsplit}
Let $G$ be a definably connected linear extension of a definably compact group $K$ by a definable torsion-free group $N$. Then the definable exact sequence
\[
1\ \longrightarrow\ N \ \stackrel{i}{\longrightarrow}\ G \stackrel{\pi}{\longrightarrow}\ K\ \longrightarrow\ 1  \\
\]

\vs \noindent
splits definably, i.e. $N$ has a definable cofactor in $G$. 
\end{prop}

\begin{proof}
First note that $N$ is the maximal normal definable torsion-free subgroup of $G$ (\ref{quocomp}).
We consider the possible cases for $K$:\\

If \underline{$K$ is abelian}, then $G$ is solvable and we have already proved in \ref{linsol} that the maximal normal definable torsion-free subgroup of $G$ has a definable cofactor. \\  

If \underline{$K$ is semisim}p\underline{le}, then $N$ is the solvable radical of $G$ (\ref{solvrad}) and by \ref{levilin} 
$G$ contains a semisimple definably connected subgroup $S$ such that $G = NS$ and $N \cap S$ is finite. 
Since $N$ is torsion-free, it follows that $N \cap S = \{e\}$ and therefore $S$ is a definable cofactor of $N$.\\

If \underline{$K$ is neither abelian nor semisim}p\underline{le}, by \cite[6.4]{hpp2} the derived subgroup $[K, K]$ of $K$ is a semisimple definably connected subgroup and $K = Z(K)^0 [K, K]$. 

Let $G_2 = \pi^{-1}([K,K])$. It is definably connected because $[K, K]$ and $N$ are definably connected. By the semisimple case, the definable exact sequence

\[
1\ \longrightarrow\ N \ \stackrel{i}{\longrightarrow}\ G_2 \stackrel{\pi_{|_{G_2}}}{\longrightarrow}\ [K, K]\ \longrightarrow\ 1  \\
\]

\vs \noindent
splits definably. Let $S$ be a definable cofactor of $N$ in $G_2$, i.e. a definable subgroup of $G_2$ definably isomorphic to $[K, K]$.
We need now a definable torus $T_1$ of $G$ definably isomorphic to $Z(K)^0$
such that $T_1S$ is a subgroup. We can find such a $T_1$ in $N_G(S)$, the normalizer of $S$ in $G$.
This is because $G_2$ is a normal subgroup of $G$ (note that $[K, K]$ is normal in $K$), and any two definable cofactors
of $N$ in $G_2$ are definable Levi subgroups of $G_2$ and they are conjugate (\ref{leviconj}), therefore by Lemma \ref{binorm} the normalizer $N_G(S)$ of $S$ in $G$
 contains some $0$-Sylow of $G$. Let $A$ be one of them.  
Consider the following diagram:

\[
\begin{array}{ccc}
G & \stackrel{\pi}{\longrightarrow}\  & K\\
\cup && \cup\\
A  & \stackrel{\pi_{|_A}}{\longleftrightarrow}   & T\\
\cup && \cup\\
T_1 & \stackrel{\pi_{|_{T_1}}}{\longleftrightarrow}\  & Z(K)^0
\end{array}
\]




\vs \noindent
Because $A$ is definably compact (\ref{lin0}), the restriction of $\pi$ to $A$ is a definable isomorphism between $A$ and a maximal definable torus $T$ of $K$ (\ref{syltor}).  

We claim that $T$ contains $Z(K)^0$. Since $Z(K)^0$ is a normal definable torus and all maximal tori
are conjugate (\ref{maxSylconj}$(d)$), it follows that $T$ (and any other maximal torus) contains $Z(K)^0$.

Define $T_1 = (\pi_{|_A})^{-1}(Z(K)^0)$. Then 
$T_1S$ is a definable subgroup of $G$ (because $T_1 \subset A \subset N_G(S)$) and it is definably isomorphic to $K$ by the canonical projection $\pi \colon G \to G/N$.

Since $N$ is torsion-free and $G/N$ is definably compact, it follows that every definable subgroup of $G$
which is definably isomorphic to $G/N$ is a definable cofactor of $N$.
\end{proof}
 
\begin{theo} \label{ctflin}
Every definably connected linear group $G$ can be decomposed as 

\[
G = KH \qquad \mbox{with} \qquad K \cap H = \{e\},
\]

\vs \noindent
for some maximal definably compact subgroup $K$ and some maximal definable torsion-free subgroup $H$. 

\end{theo}

\begin{proof}
Let $N$ be the maximal normal definable torsion-free subgroup of $G$ and let $\pi \colon G \to G/N$ be the canonical projection. By \ref{Iwagisuacca} $G/N = K_1H_1$, for some definably compact subgroup $K_1$ and some definable torsion-free subgroup $H_1$.
 
Let $G_1 = \pi^{-1}(K_1)$. By the previous proposition, the definable exact sequence

\[
1\ \longrightarrow\ N \ \stackrel{i}{\longrightarrow}\ G_1 \stackrel{\pi}{\longrightarrow}\ K_1\ \longrightarrow\ 1  \\
\]

\vs \noindent
splits definably. Let $K$ be a definable cofactor of $N$ in $G_1$ (so definably isomorphic to $K_1$) and let $H = \pi^{-1}(H_1)$. Then $H$ is torsion-free (\ref{eu}) and $G = KH$. Maximality and trivial intersection follow (see \ref{duacomptf}).
\end{proof}

\begin{rem}
For $H$ as above,
by \ref{lintorfreenoab} if $\dim H = s$ there are definable $1$-dimensional subgroups $H_1, \dots, H_s$ such that $H = H_1 \cdots H_s$. So the previous theorem gives us a definable decomposition of definably connected linear groups similar to the decomposition of connected real Lie groups we recalled in \ref{complie}.  
\end{rem}

\vs
\section{Splitting extensions} \label{splext}

We saw in \ref{comp_tor} that every definable extension of a definable torsion-free group by a definably compact group splits definably in a direct product. 

In this section we analyse
the specular case of a definable extension of a definably compact definably connected group by a definable torsion-free group. 
We see that every such a sequence splits abstractly, and it splits definably if and only if
every $0$-subgroup ($\Leftrightarrow$ some $0$-Sylow) is definably compact (\ref{extcomp}). 

To prove it we first show that when the definably compact group is semisimple, the extension splits definably
(\ref{linexsplit}). As a consequence we also find maximal semisimple definably connected definably compact subgroups in every definable group (\ref{maxsemcomp}).
 
\vs
\subsection{Definably connected extensions of a semisimple definably compact group by a definable torsion-free
group}

\begin{lem}
If the solvable radical $R$ of a definable group $G$ is definably compact, then there is no normal 
definable subgroup $H$ of $G$ such that $G/H$ is $1$-dimensional and torsion-free.
\end{lem}

\begin{proof}
We can suppose $G$ definably connected.

Let $H$ be a normal definable subgroup of $G$ such that
$G/H$ is $1$-dimensional and torsion-free. We claim that $R \subseteq H$.

If not, $R/(R \cap H)$ is an infinite torsion-free definable group (because it is definably isomorphic to $RH/H < G/H$), in contradiction with the fact that $R$ is definably compact. 

Thus $G/H$ is definably isomorphic to $(G/R)/(H/R)$. But the semisimple definably connected group $G/R$ is perfect (\cite[3.1$(v)$]{hpp2}), so $G/R$ does not have proper abelian quotients, 
in contradiction with the fact that definable $1$-dimensional torsion-free groups are abelian. 
\end{proof}

\begin{cor}\label{quoztor1}
Let $G$ be a definable group and let $N$ be its maximal normal definable torsion-free subgroup.
Then there is no normal definable subgroup $\bar{H}$ of $\bar{G} := G/N$ such that $\bar{G}/\bar{H}$ is $1$-dimensional and torsion-free.  
\end{cor}

\begin{proof}
By \ref{radgsun} and \ref{solvquozcomp} the solvable radical of $G/N$ is definably compact.
\end{proof}

\begin{prop}\label{maxtorfree1}
Let $G$ be a definably connected group and let $N$ be its maximal normal definable torsion-free subgroup. If $\dim N = 1$ then $N \subseteq Z(G)$.
\end{prop}

\begin{proof}
We use additive notation $+$ for the commutative group operation of $N$.
Let $s \colon G/N \to G$ be a definable section of the canonical projection. For every $g \in G$ there is a unique couple $(a, x)$, with $a \in N$ and $x \in G/N$, such that $g = s(x)a$. So there is a definable homomorphism

\begin{align*}
\varphi \colon G/N\ &\longrightarrow\ \Aut(N) \\
x\quad &\mapsto\ (a \mapsto s(x)as(x)^{-1}). 
\end{align*}

\vs 
This is because for every $g \in G$, the conjugation map $f_g \colon N \to N$ mapping $a \mapsto gag^{-1}$ is a definable  automorphism of $N$, so there is a homomorphism $\Phi \colon G \to \Aut(N)$, given by $g \mapsto f_g$,
and $N \subseteq \ker \Phi$, being $N$ abelian. Thus $\Phi$ induces the above definable homomorphism $\varphi \colon G/N \to \Aut(N)$ which does not depend on the choice of the section.

We want to show that $\varphi(G/N) = \{e\}$, and so $N \subseteq Z(G)$. If not trivial, $\varphi(G/N) < \Aut(N)$ is an infinite definable subgroup, since $G$ is definably connected. 

Because $N$ does not have any proper definable subgroup, by \cite[4.3]{ps00} there is a definable operation $\cdot$ over $N$, such that $(N, +, \cdot)$ is a field. 

We claim that $\Aut(N)$ is a definable group (in general it is an $\bigvee$-definable group) definably isomorphic to $(N \backslash \{0\}, \cdot)$.
Let $f \in \Aut(N)$ and let $f(1) = a \in N \backslash \{0\}$. The set $\{x \in N : f(x) = a \cdot x\}$ is a definable subgroup of $(N, +)$ containing $0$ and $1$; but $(N, +)$ does not have any proper definable subgroups, so $f(x) = a \cdot x$ for every $x \in N$. On the other hand, every definable function $N \to N$ of the form $f(x) = a \cdot x$, with $a \in N \backslash \{0\}$, is a definable automorphism of $(N, +)$, so we have done. 

Therefore, if not trivial, $\varphi(G/N)$ is definably isomorphic to $(G/N)/\ker \varphi$ which is a $1$-dimensional definable torsion-free group, contradicting our previous corollary. 
\end{proof}

\begin{rem}\label{erre2so2}
If $\dim N > 1$, the proposition above may no longer be true, even when $N$ is abelian. For $\dim N = 2$, take for instance $G = \R^2 \rtimes_{\varphi} \SO_2(\R)$, with $\varphi(A)(x) = Ax$, for every $A \in \SO_2(\R)$ and $x \in \R^2$. Then $G$ is a centerless definable group.
\end{rem}

\begin{theo} \label{linexsplit}
Let $G$ be a definably connected extension of a semisimple definably compact group $K$ by a definable torsion-free group $N$. Then the definable exact sequence 
\[
1\ \longrightarrow\ N \ \stackrel{i}{\longrightarrow}\ G \stackrel{\pi}{\longrightarrow}\ K\ \longrightarrow\ 1  
\]

\vs \noindent splits definably. Moreover any two definable cofactors of $N$ in $G$ are conjugate.
\end{theo}

\begin{proof}
Since $K$ is semisimple, it follows that $N$ is the solvable radical of $G$ (\ref{solvrad}). Thus any definable cofactor
of $N$ in $G$ is a definable Levi subgroup and they are all conjugate by \ref{leviconj}.

To find a definable cofactor of $N$ we distinguish the cases where $Z(G)$ is finite and $Z(G)$ is infinite.\\

\noindent
First suppose $Z(G)$ is \underline{finite}. We claim that $G/Z(G)$ is centerless. Otherwise: \\
\begin{enumerate}

\item[-] If the center of $G/Z(G)$ is finite and non-trivial, then its preimage in $G$ would be a normal finite subgroup and being $G$ definably connected it should be contained in $Z(G)$ (\ref{finsubincenter}), contradiction.\\

\item[-] If the center $Z$ of $G/Z(G)$ is infinite, then let $Z_1$ be the preimage in $G$ of $Z^0$
by the canonical projection $p \colon G \to G/Z(G)$.

Since $Z(G)$ is definably compact and $Z_1/Z(G) \cong Z^0$ which is torsion-free, it follows by \ref{comp_tor}
that $Z_1 = Z(G) \times N_1$, where $N_1$ is the maximal normal definable torsion-free subgroup of $Z_1$.

 Let us see that $N_1$ should be contained in $Z(G)$. For every $g \in N_1$ and $x \in G$, $gx = xgz$ for some $z \in Z(G)$. By \ref{tormaxnor}, $N_1$ is normal in $G$ as well, hence $g^{-1}x^{-1}gx \in Z(G) \cap N_1 = \{e\}$. Therefore $N_1 \subseteq Z(G)$, contradiction. \\
 
\end{enumerate}

\noindent
So we have proved that if $Z(G)$ is finite then $G/Z(G)$ is centerless.  \\

By \ref{centerl} and \ref{linsplit}, we can reduce then to the linear case. Therefore $G/Z(G) = N^{\prime} \rtimes K^{\prime}$, where $N^{\prime}$ is definably isomorphic to $N$ ($Z(G)$ finite $\Rightarrow N \cap Z(G) = \{e\}$) and $K^{\prime}$ is definably isomorphic to $K/Z(G)$. Hence the preimage of $K^{\prime}$ in $G$ is a definable cofactor of
$N$.\\

\noindent 
If $Z(G)$ is \underline{infinite}, we proceed by induction on $n = \dim N$. 
Note that since $N$ is the solvable radical of $G$, it follows that it contains $Z(G)^0$ which is infinite
as well.

\begin{itemize}

\item $n = 1$. By \ref{maxtorfree1} $N \subseteq Z(G)$ and so $N = Z(G)^0$.
Thus by \cite[5.6]{hpp2}, $[G, G]$ is a definable subgroup of $G$ such that
$G = Z(G)^0[G, G]$ and $Z(G)^0 \cap [G, G]$ is finite. Therefore $Z(G)^0 \cap [G, G] = \{e\}$ ($Z(G)^0$ is
torsion-free) and $[G, G]$ is definably isomorphic to $K$. In this case $[G, G]$ is the unique definable cofactor of $N$ in $G$ (see \ref{decder}).\\ 

\item $m < n \Rightarrow n$.  
If $Z(G)^0 = N$, see above. If $Z(G)^0 \subsetneq N$, by induction
the definable exact sequence

\[
1\ \longrightarrow\ N/Z(G)^0 \ \longrightarrow\ G/Z(G)^0 \longrightarrow\ K\ \longrightarrow\ 1  \\
\]

\vs \noindent 
(with the obvious maps)
splits definably. Let $K_1$ be a definable subgroup of $G/Z(G)^0$ definably isomorphic to $K$ and let $G_1$ be its preimage in $G$. The definable exact sequence

\[
1\ \longrightarrow\ Z(G)^0 \ \stackrel{i}{\longrightarrow}\ G_1 \stackrel{\pi}{\longrightarrow}\ K_1\ \longrightarrow\ 1  \\
\]  
 
\vs \noindent
splits definably again by induction (or by \cite[5.6]{hpp2}) and $G_1$ (so $G$) contains a definable subgroup definably isomorphic to $K_1$ (and so to $K$). 
\end{itemize}
\end{proof}

\begin{rem}
While every definable extension of a definable torsion-free group by a definably compact group is definably isomorphic to their direct product (\ref{comp_tor}),  
in general the proposition above yields a definable semidirect product that might be not direct: take for instance $\R^3 \rtimes_{\varphi} \SO_3(\R)$ with an analogous $\varphi$ as in \ref{erre2so2}.  
\end{rem}
\begin{cor} \label{0subextsem}
If $G$ is a definable extension of a semisimple definably compact group by a definable torsion-free group, then
all $0$-subgroups of $G$ are definably compact.
\end{cor}

\begin{proof}
By \ref{linexsplit} $G^0 = KN$, where $N$ is the maximal normal definable torsion-free subgroup of $G$ and $K$ is a semisimple definably compact group. Then all $0$-subgroups of $G$ are definably compact by \ref{0subctf}.
\end{proof}
\begin{prop} \label{levigsuncomp}
Let $G$ be a definably connected extension of a definably compact group $K$ by a definable torsion-free group $N$. 
Then $G$ has definable Levi subgroups, all definably isomorphic to $[K, K]$. 
\end{prop}

\begin{proof}
By \cite{hpp2}, $K = Z(K)^0 [K, K]$. Then the solvable radical $R$ of $G$ is
$\pi^{-1}(Z(K)^0)$, where $\pi \colon G \to G/N$ is the canonical projection (\ref{solvrad}).

Define $G_1 := \pi^{-1}([K, K])$. By \ref{linexsplit}, $N$ has a definable cofactor $S$ in $G_1$.
Thus $G = RG_1 = RS$, and $S$ is a definable Levi subgroup of $G$ definably isomorphic to $[K, K]$.

Since all definable Levi subgroups of $G$ are conjugate (\ref{leviconj}), the thesis follows.
\end{proof}

\vs
\subsection{Maximal semisimple definably connected definably compact subgroups}

To semplify our statements, from now on we include the trivial group $G = \{e\}$ among semisimple definable groups. With this convention, we show now that every definable group $G$ has maximal semisimple definably connected definably compact subgroups, and they are all conjugate. It turns out that each of them is definably isomorphic by the canonical projection $\pi \colon G \to G/N$ to the derived subgroup of a maximal definably compact subgroup of $G/N$, where $N$ is the maximal normal definable torsion-free subgroup of $G$.

\begin{theo} \label{maxsemcomp}
Every definable group has maximal semisimple definably connected definably compact subgroups, and they are all conjugate to each other.
\end{theo}

\begin{proof}
Let $G$ be a definable group and let $N$ be its maximal normal definable torsion-free subgroup.
We can suppose $G$ definably connected.

By \ref{Iwagisuacca} $G/N$ has maximal definably compact subgroups, all definably connected and conjugate. Let $K_1$ be one of them. Then by \cite{hpp2}, $K_1 = Z(K_1)^0 [K_1, K_1]$ and $[K_1, K_1]$ is definable and semisimple. Note that
$[K_1, K_1]$ is the unique maximal semisimple definably connected definably compact subgroup of $K_1$.

Let $\pi \colon G \to G/N$ be the canonical projection and define $G_1 = \pi^{-1}([K_1, K_1])$. By \ref{linexsplit},
the definable exact sequence

\[
1\ \longrightarrow\ N \ \stackrel{i}{\longrightarrow}\ G_1 \stackrel{\pi_{| G_1}}{\longrightarrow}\ [K_1, K_1]\ \longrightarrow\ 1  \\
\] 

\vs \noindent
splits definably. If $S_1$ is a definable cofactor of $N$ in $G_1$ (and so also a Levi subgroup of $G_1$), we claim that $S_1$ is a maximal semisimple
definably connected definably compact subgroup of $G$. 

If not, let $S$ be a semisimple definably connected definably compact subgroup containing properly $S_1$. Since $S \cap N = \{e\}$,
it follows that $\pi(S)$ is a semisimple definably connected definably compact subgroup containing properly $[K_1, K_1]$, contradiction.

Let $S_2$ be another maximal semisimple definably connected definably compact subgroup of $G$. We want to show that $S_2$ is a conjugate of $S_1$. Again, since $S_2 \cap N = \{e\}$ it follows that $\pi(S_2)$ is a maximal semisimple definably connected definably compact subgroup of $G/N$. Let $K_2$ be a maximal definably compact subgroup containing $\pi(S_2)$. Then $\pi(S_2) = [K_2, K_2]$. Define $G_2 = \pi^{-1}([K_2, K_2])$.

Since $K_2$ is a conjugate of $K_1$, it follows that $[K_2, K_2] = [K_1, K_1]^{\bar{g}}$ for some $\bar{g} \in G/N$.  
Therefore for every $g \in \pi^{-1}(\bar{g})$, denoted by $a(g)$ the conjugation map given by $x \mapsto gxg^{-1}$, the following diagram commutes:

\[
\xymatrix{
S_1 \ar[d]^{\pi_{|_{S_1}}} \ar[r]^{a(g)}  & S_1^g \ar[d]^{\pi_{|_{S_1^g}}} \ar[r]^-{i} & G_2 = N \rtimes S_2  \ar[d]^{\pi_{|_{G_2}}}\\
[K_1, K_1]  \ar[r]^{a(\bar{g})} & [K_1, K_1]^{\bar{g}} \ar@{<->}[r]^{\id} & [K_2, K_2]
}
\] 

\vs
Thus $S_1^g$ is a Levi subgroup of $G_2$ and therefore a conjugate of $S_2$ in $G_2$. This proves that $S_1$ and $S_2$ are
conjugate in $G$.
\end{proof}

\begin{rem}\label{remmaxsem}
The proof above shows that if $N$ is the maximal normal definable torsion-free subgroup of a definable group $G$ and $\pi \colon G \to G/N$ is the canonical projection, then

\begin{enumerate}
\item[$(a)$] if $S$ is a maximal semisimple definably connected definably compact subgroup of $G$, then $\pi(S) = [K, K]$, where $K$ is a maximal definably compact subgroup of $G^0/N$ containing $\pi(S)$;

\vspace{0.1 cm}
\item[$(b)$] conversely, for every maximal definably compact subgroup $K$ of $G^0/N$ there is a maximal semisimple definably connected definably compact subgroup $S$ of $G$ such that $\pi(S) = [K, K]$.
\end{enumerate}
\end{rem}

\begin{cor}\label{sylinorm}
Let $G$ be a definable group. Then:

\begin{enumerate}

\item[$(i)$] for every maximal semisimple definably connected definably compact subgroup $S$ of $G$, there is a $0$-Sylow $A$ of $G$
such that $A \subset N_G(S)$;

\vspace{0.1cm}
\item[$(ii)$] for every $0$-Sylow $A$ of $G$ there is a maximal semisimple definably connected definably compact  subgroup $S$ of $G$
such that $A \subset N_G(S)$.

\end{enumerate}

\end{cor}

\begin{proof}
We can suppose $G$ definably connected.

Note that $(i) \Leftrightarrow (ii)$ since all $0$-Sylows are conjugate (\ref{maxSylconj}) and all maximal semisimple definably connected definably compact subgroups are conjugate as well (\ref{maxsemcomp}). Indeed for every $g \in G$, 
\[
A \subseteq N_G(S)\ \Leftrightarrow\, A^g \subset N_G(S^g).
\]

\vs 
We show as $(i)$ follows from \ref{maxsemcomp} and \ref{binorm}.

Let $N$ be the maximal normal definable torsion-free subgroup of $G$ and let \\ $\pi \colon G \to G/N$ be the canonical projection. Let $K$ be a maximal definably compact subgroup of $G/N$ containing $\pi(S)$, and define $G_1:= \pi^{-1}([K, K])$. Therefore $S$ is a definable cofactor of $N$ in $G_1$, and we can apply \ref{binorm} to find a $0$-Sylow of $G_2: = \pi^{-1}(K)$ in $N_G(S)$. Note that $E(G/G_2) = E(H) = \pm 1$ (where $H$ is a torsion-free definable complement of $K$ in $G/N$: \ref{Iwagisuacca}), so any $0$-Sylow of $G_2$ is a $0$-Sylow of $G$ as well (\ref{charsyl}). 
\end{proof} 
\vs
\subsection{Definably connected extensions of a definably compact group by a definable torsion-free group} 

\begin{fact} \emph{(\cite[6.11]{zero}, \cite[1.2]{div}).} \label{compunmax}
Let $G$ be a definably connected definably compact group, and let $T$ be a maximal definable torus
of $G$ $($i.e. $T$ is a $0$-Sylow of $G)$. Then
\[
G\ =\ \bigcup_{g \in G}T^g.
\]

\end{fact}

\begin{theo}\label{extcomp}
Let $G$ be a definably connected extension of a definably compact group $K$ by a  definable torsion-free group $N$. Then \\

\begin{enumerate}

\item[-] the definable exact sequence 

\[
1\ \longrightarrow\ N \ \stackrel{i}{\longrightarrow}\ G \stackrel{\pi}{\longrightarrow}\ K\ \longrightarrow\ 1  \\
\]

\vs \noindent splits abstractly, and it splits definably if and only if every $0$-subgroup of $G$ is definably compact;\\

\item[-] for every $0$-Sylow $A$ of $G$ and for every direct complement $T$ of $A \cap N$ in $A$ $($see \ref{factor}$)$, there is 
a cofactor $K_T$ of $N$ in $G$ such that

\[
K_T = \bigcup_{x \in K} T^x;
\]

\vs
\item[-] the derived subgroup $[K_T, K_T]$ of $K_T$ is definable and it is a maximal semisimple definably connected definably compact subgroup of $G$.

\end{enumerate}
\end{theo}

\begin{proof}
Let us consider the possible cases for $K$:\\

If \underline{$K$ is abelian}, then $G$ is solvable and the thesis follows by \ref{solvsempr}. In this case of course $K_T = T$ and $[K_T, K_T] = \{e\}$.\\

If \underline{$K$ is semisim}p\underline{le}, then the extension above splits definably (\ref{linexsplit}) and every $0$-subgroup of $G$ is definably compact (\ref{0subextsem}). Thus $A \cap N = \{e\}$.

Note that every maximal semisimple definably connected definably compact subgroup of $G$ is also a maximal definably compact subgroup of $G$.

Let $S$ be one of them such that $A \subseteq N_G(S)$
(\ref{sylinorm}$(ii)$). Therefore $AS$ is a definably compact subgroup of $G$. By maximality of $S$ it follows that $A \subset S$. By \ref{compunmax}, 
\[
K_A: = S = \bigcup_{x \in K_A} A^x. 
\]

\vs 
Moreover by \cite{hpp2} $K_A = [K_A, K_A]$.\\  
 
If \underline{$K$ is neither abelian nor semisim}p\underline{le}, then by \cite{hpp2} $K = Z(K)^0[K, K]$, where $[K,K]$ is definable and semisimple.

Let $A$ be a $0$-Sylow of $G$, and $T$ a direct complement of $A \cap N$ in $A$. We want to find a cofactor $K_T$ of $N$ in $G$
of the form

\[
K_T = \bigcup_{x \in K_T} T^x. 
\]

\vs \noindent
By \ref{maxsemcomp}, \ref{remmaxsem} and \ref{sylinorm} there is a maximal semisimple definably connected definably compact subgroup $S$ of $G$ definably isomorphic to $[K, K]$ (by $\pi_{|_S} \colon S \to [K, K]$) such that $A \subset N_G(S)$. In particular
$T \subset N_G(S)$ and therefore $K_T := TS$ is a subgroup. It is easy to check that $\pi_{|_{K_T}} \colon K_T \to K$ is an isomorphism, 
$[K_T, K_T] = S$, and $\pi(A) = \pi(T)$ is a maximal torus of $K$ (\ref{syltor}). By \ref{compunmax},
\[
K = \bigcup_{x \in K}\pi(T)^x,
\]

and therefore 

\[
K_T = \bigcup_{x \in K_T}T^x.
\]

If $A$ is definably compact (and therefore every $0$-subgroup of $G$ is), then $N \cap A = \{e\}$, and $K_T$ is definable. 

Otherwise,  
suppose that there is a definable cofactor of $N$ in $G$. Then all $0$-subgroups of $G$ are definably compact (\ref{0subctf}), in contradiction with the fact that $A$ is not.
\end{proof}
\subsection{An abstract decomposition}

As a consequence of previous results about splitting extensions we obtain the following abstract decomposition of definably connected groups:

\begin{theo} \label{decab}
Let $G$ be a definably connected group and $N$ its maximal normal definable torsion-free subgroup.\\

\begin{enumerate}
\item[$(a)$]
For every $0$-Sylow $A$ of $G$ and for every direct complement $T$ of $A \cap N$ in $A$, there are
a subgroup $K$ of $G$ and a maximal definable torsion-free subgroup $H$ of $G$
such that

\[
K = \bigcup_{x \in K} T^x, \quad G = KH, \quad K \cap H = \{e\}.
\]

\vs

The derived subgroup $[K, K]$ of $K$ is definable, and it is a maximal semisimple definably connected definably compact subgroup of $G$.\\

\item[$(b)$] For every maximal semisimple definably connected definably compact subgroup $S$ of $G$ there are a
subgroup $K$ of $G$ and a maximal definable torsion-free subgroup $H$ of $G$ such that

\[
S = [K, K], \qquad G = KH, \qquad K \cap H = \{e\}.
\]

\vs \noindent
The subgroup $K$ is the union of conjugates of a complement of $A \cap N$ in $A$, for some $0$-Sylow $A$ of $G$.\\

\item[$(c)$] In both $(a)$ and $(b)$ above the restriction to $K$ of the canonical projection $\pi \colon G \to G/N$ is an isomorphism with a maximal definably compact subgroup of $G/N$, and $K$ is definable if and only if $A$ is definably compact $($$\Leftrightarrow$ every $0$-subgroup of $G$ is definably compact$)$.

\end{enumerate}
\end{theo}

\begin{proof} \hspace{7cm}
\begin{enumerate} 
\item[$(a)$] 
 
Let $K_1$ be a maximal definably compact subgroup of 
$G/N$ containing $\pi(A)$ and let $H_1$ be a definable torsion-free subgroup such that $G/N = K_1H_1$ (\ref{Iwagisuacca}). 

Let $G_1 = \pi^{-1}(K_1)$. Note that $A \subseteq G_1$. 
By \ref{extcomp} there is an abstract cofactor $K$ of $N$ in $G_1$ such that

\[
K = \bigcup_{x \in K} T^x,
\]

\vs \noindent
and $[K, K]$ is a maximal semisimple definably connected definably compact subgroup of $G_1$ (and of $G$ as well).

Let $H = \pi^{-1}(H_1)$; then $G = KH$ and $K \cap H = \{e\}$. As we showed in \ref{maxtorfreenon}, $H$ is a maximal definable torsion-free subgroup of $G$.\\  

\item[$(b)$] Let $K_1$ be a maximal definably compact subgroup of $G/N$ containing $\pi(S)$, and let $H_1$ be a torsion-free definable complement of $K_1$ in $G/N$ (\ref{Iwagisuacca}). 

Let $H = \pi^{-1}(H_1)$ and $G_1 = \pi^{-1}(K_1)$. Then $\pi(S) = [K_1, K_1]$ (\ref{remmaxsem}) and by Theorem \ref{extcomp} applied to $G_1$, for every $0$-Sylow $A$ of $G_1$ and for every cofactor $T$ of $A \cap N$ in $A$ there is an abstract
subgroup $K$ of $G$ with the required properties. Because
\[
E(G/G_1) = E(H) = \pm 1,
\]

\vs \noindent
$A$ is a $0$-Sylow of $G$ as well (\ref{charsyl}).\\

\item[$(c)$] In both cases $G = KH$, where $H = \pi^{-1}(H_1)$ for some maximal definable torsion-free
subgroup of $G/N$ such that $G/N = K_1H_1$, where $K_1$ is a maximal definably compact subgroup of $G/N$.
Since $K \cap N = \{e\}$, it follows that $\pi_{|_K} \colon K \to K_1$ is an isomorphism. 

By \ref{extcomp} applied to $G_1 : = \pi^{-1}(K_1) = KN$, $K$ is definable if and only if every $0$-subgroup
of $G_1$ is definably compact. Since $E(G/G_1) = E(H_1) = \pm 1$, every $0$-Sylow of $G_1$ is a $0$-Sylow of $G$
as well (\ref{charsyl}) and the proof is completed.
\end{enumerate}
\end{proof}


\begin{cor} \label{torder}
Let $G$, $N$, $A$, $T$, $K$ be as in \ref{decab}$(a)$. Then

\[
K = T [K, K].
\]
\end{cor}

\vs
\begin{proof}
Let $K_1$ be the maximal definably compact subgroup of $G/N$ such that 
\[
\pi_{|_K} \colon K \ \longrightarrow \ K_1
\]

\noindent
is an isomorphism \ref{decab}$(c)$. Then $T_1 : = \pi(A)$ is a maximal definable torus of $K_1$ (\ref{syltor}) and it contains $Z(K_1)^0$ (which is a normal definable torus, so it is contained in every maximal definable torus, because 
they are all conjugate to each other). 

Since $K_1 = Z(K_1)^0[K_1, K_1]$, it follows that $K_1 = T_1[K_1, K_1]$ (note that $T_1 = \pi(T)$)
and then $K = T[K, K]$. 
\end{proof}

\vs
\section{A definable Lie-like decomposition}

In this section we find a definable \ominimal analogue of the decomposition in \ref{complie} of connected real Lie groups, where maximal tori are replaced by maximal $0$-subgroups.

\begin{prop} \label{smdefK}
Let $G$ be a definably connected group and $N$ its maximal normal definable torsion-free subgroup. Let $A$ be a $0$-Sylow of $G$, $T$
a cofactor of $A \cap N$ in $A$, and 
\[
G = KH
\] 

\vs \noindent
an abstract decomposition with $K$ and $H$ as in \ref{decab}$(a)$, i.e. $K = \bigcup_{x \in K}T^x$ and $K \cap H = \{e\}$.

Then the smallest definable subgroup $P$ of $G$ containing $K$ is such that:\\

\begin{enumerate}

\item
$P = AS$, where $S = [K, K]$ is a maximal semisimple definably connected definably compact subgroup of $G$;\\

\item
$P = \bigcup_{x \in P}A^x$;\\

\item
$A \cap N$ is the unique maximal definable torsion-free subgroup of $P$ and $P/(A \cap N)$ is definably compact;\\

\item $Z(P)^0 = C_P(S)^0 = (A \cap N) \times (\pi_{|_K})^{-1}(Z(\pi(K))^0)$ is the solvable radical of $P$
and $S = [K, K]$ is a definable Levi subgroup of $P$;\\

\item
$P = (A \cap N) \times K$;\\

\item
$Z(P) = (A \cap N) \times Z(K)$. \\

\end{enumerate}
\end{prop}

\begin{proof}
Let $\pi \colon G \to G/N$ be the canonical projection. 
Define $G_1 := KN$ and observe that  $G_1 = \pi^{-1}(K_1)$, for some maximal definably compact subgroup $K_1$ of $G/N$ and 
\[
\pi_{|_K} \colon K \ \longrightarrow \ K_1
\]

\vs \noindent
is an isomorphism (\ref{decab}$(c)$). So $P \subseteq G_1$ and $\pi_{|_P} \colon P \to K_1$ is a definable surjective homomorphism.

We consider the possible cases for $K_1$. \\  

If $K_1$ is \underline{abelian}, then $K = T$ and we claim that the smallest definable subgroup $P$ of $G$ containing $T$ is $A$.

If not,  
$T \subseteq P \subsetneq A$. Since $A = T \times (A \cap N)$, it follows that $A/P$ is definably isomorphic
to $(A \cap N)/(P \cap N)$, so $E(A/P) = E((A \cap N)/(P \cap N)) = \pm 1$, in contradiction with the fact
that $A$ is a $0$-group.

The other statements follow trivially.\\

If $K_1$ is \underline{semisim}p\underline{le}, then $K$ is definable (\ref{linexsplit}), so $P = K$. In this case $A$ is definably compact (\ref{0subctf}) and $K = P$ is the union of conjugates of $A$ (\ref{decab}). By \cite{hpp2} $K = [K, K]$ and the other statements are obvious. \\

Suppose now $K_1$ is \underline{neither abelian nor semisim}p\underline{le}.

\begin{enumerate}

\item
We want to show that $P = AS$, where $S:= [K, K]$.

We claim that $A \subset N_G(S)$. Since $S \lhd K$ and $T \subset K$, of course $T \subset N_G(S)$.
As we showed above, the smallest definable subgroup containing $T$ is $A$, so $A \subset N_G(S)$ as well.
Therefore $AS$ is a definable subgroup and $P \subseteq AS$. 
On the other hand, by \ref{torder} $K = TS$, and since $A$ is the smallest definable subgroup containing $T$, it follows that $A \subset P$, so $AS \subseteq P$ and $AS = P$.\\

\item
We want to show that

\[
P\ =\ \bigcup_{x \in P}A^x.
\]  

\vs 
\begin{enumerate}

\item[$(\subseteq)$] Since $P = AS$ and $P$ is a subgroup, it is enough to verify that $S \subset \bigcup_{x \in P}A^x$. We claim that

\[
S\ =\ \bigcup_{x \in S}(A \cap S)^x.  
\]

\vs \noindent
Actually we can show that $(A \cap S)^0$ is a $0$-Sylow of $S$, from which the claim follows by \ref{compunmax}. Since $(A \cap S)^0$ is definably connected (obviously), definably compact (it is a definable subgroup of $S$ which is definably compact) and abelian (it is a subgroup of $A$ which is abelian), it follows that it is a definable torus (\ref{torusis0gr}) and then a $0$-group. It is also a $0$-Sylow of $S$ by \ref{maxSylconj}, because
 
\[
E(S/(A \cap S)^0) = E(S/(A \cap S)) E((A \cap S)/ (A \cap S)^0) \neq 0,
\]

\vs \noindent
since $S/(A \cap S)$ is in definable bijection with $P/A$ and $E(P/A) \neq 0$, being $A$ a $0$-Sylow. 

\vs
\item[$(\supseteq)$] $A \subseteq N_{G}(S)\ \ \Rightarrow\ A^y \subseteq AS = P \ \ \forall\, y \in S\ \Rightarrow\ A^x\ \subseteq P\ \forall\ x \in P$.\\
\end{enumerate}

\item $P = AS = (A \cap N)TS = (A \cap N)K$ (\ref{torder}). Therefore the kernel of the map
\[
\pi_{|_P} \colon P\ \longrightarrow\ K_1
\]  

\vs \noindent
is $A \cap N$. This shows that $P \cap N = A \cap N$ and thus $A \cap N \lhd P$. 
Since $P/(A \cap N)$ is definably compact, $A \cap N$ is the maximal (normal) definable torsion-free subgroup of $P$ (\ref{quocomp}).\\

\item 
Since $P/\pi^{-1}(Z(K_1))$ is semisimple (being definably isomorphic to a quotient of $[K_1, K_1]$),
it follows that $\pi^{-1}(Z(K_1))^0 = (A \cap N) \rtimes (\pi_{|_K})^{-1}(Z(K_1)^0)$ is the solvable radical of $P$
(\ref{solvrad}). Call it $R_P$.

We claim that $R_P = C_P(S)^0$. First we show that $A \cap N \subset C_P(S)$. 
Let $a \in A \cap N \subset N_G(S)$. For every $s \in S$, we have
$asa^{-1}s^{-1} \in N \cap S = \{e\}$. Therefore $R_P \subset C_P(S)^0$.

For every $x \in C_P(S)$, $\pi(x) \in C_{K_1}([K_1, K_1]) = Z(K_1)$. Thus $R_P = C_P(S)^0$.

We claim that actually $R_P \subset Z(P)$. Note that each $x \in R_P$ can be written as $x = a_1y$, with $a_1 \in A \cap N$ and $y \in Z(K)$, and each $p \in P$ can be written as $p = a_2ts$ (see $(3)$), with $a_2 \in A \cap N$, $t \in T$ and $s \in S$. So $xp = px$ for every $x \in R_P$ and every $p \in P$.

Therefore $R_P = Z(P)^0$ and $P = R_PS$, so $S$ is a Levi subgroup of $P$.\\ 

\item By $(4)$, $P = R_PS = (A \cap N)K$. We have seen in $(3)$ that $A \cap N$ is normal in $P$. Also $K$
is normal in $P$, since $A \cap N \subset Z(P)$.\\
 
\item It is immediate from $(4)$ and $(5)$.

\end{enumerate}
\end{proof}

\begin{theo} \label{lielike}
Let $G$ be a definably connected group and $N$ its maximal normal definable torsion-free subgroup. Let $A$ be a $0$-Sylow of $G$. Then there are definably connected subgroups $P, H < G$,
$H \supseteq N$ $($maximal$)$ torsion-free, such that

\[
G = PH  \qquad \mbox{and} \qquad P = \bigcup_{x \in P}A^x,
\]

\noindent
where $P \cap H = P \cap N = A \cap N$
is the maximal $($normal$)$ definable torsion-free subgroup of $P$.
\end{theo}
 
\begin{proof} 
See \ref{smdefK}. Note that $P$ is definably connected since $A^x$ is definably connected
for every $x \in P$, and therefore $A^x \subset P^0$ for every $x \in P$.
Being torsion-free, also $H$ is definably connected (\ref{torfree}$(a)$).

Moreover $P \cap H = P \cap N$, since $P \cap N$ is the maximal definable torsion-free definable subgroup of $P$.
\end{proof}


\vs
\section{Homotopy equivalence}

In this last section we show (as a consequence of results in \cite{ps05} and in the previous sections)
that every definably connected group $G$ in an \ominimal expansion of a field is definably homotopy equivalent to any maximal definably compact subgroup of $G/N$, where $N$ is the maximal normal definable torsion-free subgroup of $G$ (\ref{homcompmax}).\\

 So let  $\M = \langle M, <, +, \cdot, \dots \rangle$ be an \ominimal expansion of a real closed field. Definable groups of this section are supposed to be definable in \M.

We recall first the definitions of definable homotopy and homotopy equivalence:

Let $X, Y$ be definable sets and $f,g
\colon X \to Y$ definable continuous maps. A {\bf definable
homotopy between} $f$ and $g$ is a definable continuous map $\mathcal{H} \colon
X \times [0, 1] \to Y$ such that $f(x) = \mathcal{H}(x,0)$ and
$g(x) = \mathcal{H}(x, 1)$ for every $x \in X$.\\
A definable set $X$ is called {\bf definably contractible} to the
point $\bar{x} \in X$ if there is a definable homotopy $\mathcal{H} \colon X
\times [0, 1] \to X$ between the identity map on $X$ and the map
$X \to X$ taking the constant value $\bar{x}$. \\

\noindent
Two definable sets
$X$ and $Y$ are {\bf definably homotopy equivalent} if there
are definable continuous maps $f \colon X \to Y$, $g \colon Y
\to X$ such that there exist definable homotopies $\mathcal{H}_X$ between $(f
\circ g) \colon X \to X$ and the identical map on $X$, and $\mathcal{H}_Y$
between $(g \circ f) \colon Y \to Y$ and the identical map on $Y$.

\begin{fact} \emph{(\cite[5.7, 5.1]{ps05}).}
\begin{enumerate}

\item[$(a)$] Every $n$-dimensional definable torsion-free group $H$ is definably homeomorphic to
$M^n$ $($and so definably contractible$)$.

\item[$(b)$] For every definably contractible definable subgroup $H$ of a definable group $G$, there is a definable continuous section $s \colon G/H \to G$. 

\end{enumerate}
\end{fact}

\begin{cor}\label{homeo}
Let $G$ be a definable group and let $N$ be a normal definable torsion-free subgroup of $G$. Then $G$ is definably homeomorphic to $G/N \times N$.
\end{cor}

\begin{proof}
Let $s \colon G/N \to G$ be a definable continuous section
of the canonical projection $\pi \colon G \to G/N$. Then the map

\begin{align*}
G  \quad &\longrightarrow\quad  G/N \times N \\
g  \quad &\mapsto\quad   (\ \pi(g),\ s(\pi(g))^{-1}g\ )
\end{align*}

\vspace{0.3cm} \noindent is a definable continuous bijection with definable and continuous
inverse

\begin{align*}
G/N\ \times\ N  \quad &\longrightarrow \quad  \ G\\
(\ \pi(g)\ ,\ x\ )  \quad &\mapsto \quad   s(\pi(g))x.
\end{align*}

\end{proof}

\begin{prop}\label{solvhom0syl}
Every solvable definably connected group is definably homotopy equivalent to any of its $0$-Sylow.
\end{prop}

\begin{proof}
Let $G$ be a solvable definably connected group and let $N$ be its maximal normal definable torsion-free subgroup. By \ref{homeo} $G$ is definably homeomorphic to $G/N\times N$. On the other hand
$G = AN$, for any $0$-Sylow $A$ of $G$ (\ref{prodsolv}). If $N_A$ is the maximal definable torsion-free subgroup of $A$, then again
$A$ is definably homeomorphic to $A/N_A \times N_A$. Because $N$, $N_A$ are definably contractible and $G/N$ and $A/N_A$ are definably isomorphic, we get that both $G$ and $A$ are definably homotopically equivalent to $G/N$.
\end{proof}

\begin{rem}
The proposition above can be viewed as an \ominimal analogue of the fact that avery connected solvable Lie group is homotopy equivalent
to any of its maximal torus.
\end{rem}
\vs 
\begin{theo} \label{homeonocomp}
Let $G$ be a definably connected group, $N$ its maximal normal definable torsion-free subgroup and $K$ a maximal definably compact subgroup of $G/N$. Then $G$ is definably homeomorphic to $K \times M^s$, where $s \in \N$ is the maximal dimension of a definable torsion-free subgroup of $G$.  
\end{theo} 
 
\begin{proof}
By \ref{homeo}, $G$ is definably homeomorphic to $N \times G/N$. By \ref{Iwagisuacca}, $G/N = KH$, with $K$ definably compact and $H$ definable torsion-free subgroups. Thus $G/N$ is definably homeomorphic to $K \times H$ and $G$ is definably homeomorphic to $K \times M^s$, $s = \dim N + \dim H$. 

We claim that $s$ is the maximal dimension of a definable torsion-free subgroup of $G$. First note that
if $\pi \colon G \to G/N$ is the canonical projection, then $\pi^{-1}(H)$ is a definable torsion-free subgroup
with dimension equal to $s$.

Suppose $\dim H_1 > s$, for some $H_1$ definable torsion-free subgroup of $G$. Then $\pi(H_2) \cap K = \{e\}$ and $\dim \pi(H_2) + \dim K > \dim H + \dim K = \dim G/N$, contradiction.  
\end{proof}

\begin{cor}\label{homcompmax}
Every definably connected group $G$ is definably homotopy equivalent to each maximal definably compact subgroup of $G/N$, where $N$ is the maximal normal definable torsion-free subgroup of $G$.
\end{cor}
\begin{cor}
A definably connected group is definably contractible if and only if it is torsion-free.
\end{cor}

\begin{proof}
The $(\Leftarrow)$ implication is \cite[5.7]{ps05}. For the other one, by \ref{homcompmax}
it is enough to show that a definably compact definably connected group is definably contractible if and only if it is trivial. This is a consequence of Theorem 3.7 in \cite{bmo} which proves that the homotopy groups of a definably compact definably connected group $K$ are isomorphic to the homotopy groups of the connected compact Lie group $K/K^{00}$.
\end{proof}

\vs
\section{Summary}  

In this last section we give a brief overview of the main results proved above, hoping to make
the picture more clear. \\

Many of the assertions of the following theorem were actually already known. We put them here
to stress the structure of a definable group in terms of certain particular normal (maximal) definable subgroups
and related quotients:

\begin{theo}\label{struc}
Let $G$ be a definable group. Then there exists
a sequence of characteristic $($i.e. invariant by definable automorphisms of $G$$)$ definable subgroups

\[
N\ \lhd\ R\ \lhd\ Z\ \lhd\ G^0\ \lhd\ G
\]
such that:
\begin{enumerate}

\vspace{0.1cm}
\item  $G^0$ is the definably connected component of the identity, 

\vspace{0.1cm}
\item  $Z$ is the maximal normal solvable definable subgroup of $G^0$,

\vspace{0.1cm}
\item  $R$ is the solvable radical of $G$ and $R = Z^0$, 

\vspace{0.1cm}
\item  $N$ is the maximal normal definable torsion-free subgroup of $G$,

\vspace{0.1cm}
\item  $G/G^0$ is finite,

\vspace{0.1cm}
\item  $G^0/Z$ is a direct product of definably simple groups, 

\vspace{0.1cm}
\item  $Z/R$ is the finite center of the semisimple definably connected group $G^0/R$,

\vspace{0.1cm}
\item  $R/N$ is a definable torus, and the solvable radical of $G/N$,

\vspace{0.1cm}
\item  $G^0/N = KH$, for some $H$ definable torsion-free and $K$ definably compact.
\end{enumerate}
\end{theo}

\begin{proof} \hspace{7cm}
 
$(1)$ and $(5)$ are \cite[2.12]{pi1}.

\vspace{0.1cm}
$(3)$ is \ref{solvrad}.

\vspace{0.1cm}
$(4)$ is \ref{unifree}.

\vspace{0.1cm}
$(2)$ and $(7)$: since $G^0/R$ is semisimple, it follows that its center is finite. Let $Z$ be its preimage
in $G^0$. Then $Z$ is solvable and we claim that it is the maximal normal solvable definable subgroup of $G^0$.
If not, let $Z \subsetneq Z_1 \lhd G^0$, with $Z_1$ definable and solvable. Then $Z_1/R \lhd G^0/R$ is solvable and by \ref{semisnosolv} it cannot be infinite, so $Z_1/R = Z/R$ because every finite normal subgroup
of a definably connected group is contained in its center (\ref{finsubincenter}).

\vspace{0.1cm}
$(6)$ is \cite[4.1]{pps1} (see \ref{theosem}).

\vspace{0.1cm}
$(8)$ see \ref{solvquozcomp}, \cite[5.4]{ps00} and \ref{radgsun}.

\vspace{0.1cm}
$(9)$ see \ref{maxcomprad}.
\end{proof}

\begin{cor}
Every definably connected group $G$ in an \ominimal \hspace{-0.1cm}expansion of a real closed field $\M = \langle M, <, +, \cdot, \dots \rangle$ is definably homeomorphic to

\[
K \times M^s,
\]

\vs \noindent
where $K$ is any maximal definably compact subgroup of $G/N$, and $s \in \N$ is the maximal dimension of a definable torsion-free subgroup of $G$.
\end{cor}

\begin{proof}
It follows from \ref{struc}$(9)$ and \cite{ps05}. See \ref{homeonocomp}. 
\end{proof}

\begin{cor}
Every definable group has maximal definable torsion-free subgroups.
\end{cor}

\begin{proof}
Again it is a consequence of \ref{struc}$(9)$ and \cite{ps05}. See \ref{maxtorfreenon}.
\end{proof}

\vs
In terms of definable extensions, Theorem \ref{struc} turns into the following list of definable exact sequences:\\

\begin{enumerate}

\item[-]
Every definable group is a definable extension of a finite group by a definably connected group:

\[
1\ \longrightarrow\ G^0 \ \stackrel{i}{\longrightarrow}\ G\ \stackrel{\pi}{\longrightarrow}\ G/G^0\ \longrightarrow\ 1.  
\]

\vs \item[-]
Every definably connected group is a definable extension of a direct product of definably simple groups
by a solvable definable group:

\[
1\ \longrightarrow\ Z \ \stackrel{i}{\longrightarrow}\ G^0\ \stackrel{\pi}{\longrightarrow}\ G^0/Z\ \longrightarrow\ 1.  
\]

\vs \item[-]
Every definably connected group is a definable extension of a semisimple definably connected group by a solvable definably connected group: 

\[
1\ \longrightarrow\ R \ \stackrel{i}{\longrightarrow}\ G^0\ \stackrel{\pi}{\longrightarrow}\ G^0/R\ \longrightarrow\ 1.  
\] 

\vs \item[-]
Every definable group is a definable extension of a definable group with definably compact solvable radical by a definable torsion-free group: 

\[
1\ \longrightarrow\ N \ \stackrel{i}{\longrightarrow}\ G\ \stackrel{\pi}{\longrightarrow}\ G/N\ \longrightarrow\ 1.  
\]

\vs \item[-]
Every solvable definable group is a definable extension of a definably compact group by a definable torsion-free group:

\[
1\ \longrightarrow\ N \ \stackrel{i}{\longrightarrow}\ Z\ \stackrel{\pi}{\longrightarrow}\ Z/N\ \longrightarrow\ 1.  
\]

\vs \item[-]
Every solvable definably connected group is a definable extension of a definable torus by a   definable torsion-free group: 

\[
1\ \longrightarrow\ N \ \stackrel{i}{\longrightarrow}\ R\ \stackrel{\pi}{\longrightarrow}\ R/N\ \longrightarrow\ 1.  
\]

\end{enumerate}

\vs \noindent
We remark that all groups involved in the list are \textbf{maximal} with the mentioned properties.
\vs
\begin{theo}
Let $G$ be a definably simple group. Then 

\begin{enumerate}

\item[$(a)$]

there is a definable real closed field \Rs\ and some $m \in \N$ such that $G$ is definably isomorphic to a definable group $G_1 < \GL_{m}(\Rs)$, definable in \Rs, with the following properties:

\begin{enumerate}
\vspace{0.1cm}
\item[$(i)$] $G_1 = KH$, with $K = G_1 \cap O_{m}(\Rs)$ and $H = G_1 \cap T^+_{m}(\Rs)$;  
 
\vspace{0.1cm}
\item[$(ii)$] $H = AN$, with $A = G_1 \cap D^+_{m}(\Rs)$ and $N = G_1 \cap UT_{m}(\Rs)$; \\
\end{enumerate}

\item[$(b)$] $G$ has maximal definably compact subgroups, all definably connected and conjugate to each other, and every such a maximal definably compact subgroup has a definable torsion-free complement.

\end{enumerate}
\end{theo}

\begin{proof}
See \ref{declin} and \ref{conjmax}.
\end{proof}
\vs
\begin{prop}
Let $G$ be a definably connected group and let $N$ be its maximal normal definable torsion-free subgroup.

\begin{enumerate}
\vspace{0.1cm}
\item[$(a)$] If $G/Z(G)$ or $G/N$ is definably compact, then $G$ has a definable Levi decomposition.

\vspace{0.1cm}
\item[$(b)$] If $G$ has a definable Levi decomposition, then all definable Levi subgroups are conjugate to each other. 
\end{enumerate}
\end{prop}

\begin{proof} \hspace{7cm}
\begin{enumerate}
\item[$(a)$] See \cite{hpp2}, \ref{leviquozcomp} and \ref{levigsuncomp}.

\item[$(b)$] See \ref{leviconj}.
\end{enumerate}
\end{proof}
\begin{theo}
Let $G$ be definably connected, $K$ definably compact, $N$ definable torsion-free groups. Then

\begin{enumerate}
\item[$(i)$] Every definable exact sequence

\[
1\ \longrightarrow\ K \ \stackrel{i}{\longrightarrow}\ G\ \stackrel{\pi}{\longrightarrow}\ N\ \longrightarrow\ 1  \\
\]

\vs \noindent
splits definably in a direct product. Namely $G$ is definably isomorphic to 

\[
K \times N.
\]

\vs
\item[$(ii)$] Every definable exact sequence

\[
1\ \longrightarrow\ N \ \stackrel{i}{\longrightarrow}\ G\ \stackrel{\pi}{\longrightarrow}\ K\ \longrightarrow\ 1  \\
\]

\vs \noindent
splits abstractly, i.e. $G$ is abstractly isomorphic to a semidirect product

\[
N \rtimes K.
\]

\vs \noindent
It splits definably if and only if every $0$-subgroup of $G$ is definably compact $($in particular, the last condition holds when $K$ is semisimple or $G$ is linear$)$.

\end{enumerate} 

\end{theo}

\begin{proof} \hspace{7cm}
\begin{itemize}
\item[$(i)$] If $\dim N = 1$, it is proved in \cite[5.1]{solv}. If $\dim N > 1$, see \ref{comp_tor}.\\

\item[$(ii)$] See \ref{extcomp}.
\end{itemize}
\end{proof}

\begin{theo}
Let $G$ be a definably connected group, $N$ its maximal normal definable torsion-free subgroup
and $\pi \colon G \to G/N$ the canonical projection. \\

$\bullet$ For every $0$-Sylow $A$ of $G$: \\

\begin{enumerate}

\item $A \cap N$ is the maximal definable torsion-free subgroup of $A$ and $\pi(A)$ is a maximal definable torus of $G/N$.\\

\item For every direct complement $T$ of $A \cap N$ in $A$, there is a subgroup $K$ of $G$
such that\\

\begin{enumerate}

\item
\[
K = \bigcup_{x \in K} T^x, \quad G = KH, \quad K \cap H = \{e\},
\]

\vs \noindent
where $H$ is a maximal definable torsion-free subgroup of $G$;\\

\item the derived subgroup $[K, K]$ of $K$ is definable, it is a maximal semisimple definably connected definably compact subgroup of $G$, and any maximal semisimple definably connected definably compact subgroup of $G$ is a conjugate of $[K, K]$;\\

\item the restriction to $K$ of the canonical projection $\pi \colon G \to G/N$ is an abstract isomorphism between $K$ and a maximal definably compact subgroup of $G/N$;\\ 

\item One can find such a $K$ definable if and only if $A$ is definably compact $($$\Leftrightarrow$ every $0$-subgroup of $G$ is definably compact$)$;\\

\item the smallest definable subgroup of $G$ containing $T$ is $A$, and the smallest definable subgroup $P$ containing $K$ is such that:

\begin{enumerate}
\vspace{0.1cm}
\item
\[
P = \bigcup_{x \in P}A^x;
\]

\vs
\item the maximal $($normal$)$ definable torsion-free subgroup of $P$ is
\[
P \cap H = P \cap N = A \cap N
\]

\noindent
and $P/(A \cap N)$ is definably compact;\\

\item
$P = AS$, where $S = [K, K]$; \\ 

\item $P = (A \cap N) \times K$;\\ 

\item $Z(P)^0 = C_P(S)^0 = (A \cap N) \times (\pi_{|_K})^{-1}(Z(\pi(K))^0)$ is the solvable radical of $P$,  
and $S$ is a definable Levi subgroup of $P$; \\

\item $Z(P) = (A \cap N) \times Z(K)$.\\

 \end{enumerate}
\end{enumerate}

\item $G$ is solvable $\ \Leftrightarrow\  P = A \ \Leftrightarrow\ K = T$. \\
\end{enumerate}

$\bullet \, $ For every maximal semisimple definably connected definably compact subgroup $S$ of $G$
there is a subgroup $K$ of $G$ with $S = [K, K]$, and there are a $0$-Sylow $A$ of $G$ and a direct complement $T$ of $A \cap N$ in $A$ such that $(a)$, $(c)$, $(d)$, $(e)$ above hold.
\end{theo}

\begin{proof}
See \ref{maxtfA}, \ref{syltor}, \ref{remmaxsem}, \ref{decab}, \ref{smdefK}.
\end{proof}

\vs \noindent
We conclude with a little scheme about the parallelism between definable groups in \ominimal structures
and real Lie groups:

\vs \hspace{-1cm}
\begin{tabular}{ccc}
\textsc{Definable groups} &  & \textsc{Lie groups}\\
\\
definable homomorphisms & $\ \longleftrightarrow\ $ & smooth homomorphisms\\
\\
definable subgroups& $\ \longleftrightarrow\ $ &  closed subgroups  \\
\\
solvable radical  & $\ \longleftrightarrow\ $ & solvable radical \\
\\
 $0$-subgroups &$\longleftrightarrow$& tori\scshape  \\
\\
 $0$-Sylows &$\longleftrightarrow$& maximal tori\scshape  \\
\\
 max. subgr. which is union of $0$-Sylows &$\longleftrightarrow$& max. compact subgr. \\
$||$ & & $||$\\
union of the conj. of a $0$-Sylow&&union of the conj. of a max. torus 

\end{tabular}

\vs 

\end{document}